\documentclass[12pt]{amsart}
\usepackage{amsmath,amsthm, amssymb,mathrsfs}
\usepackage{graphics}
\usepackage{amsfonts} 
\usepackage[usenames,dvipsnames]{color}
\usepackage{tikz}
\usepackage{tkz-euclide}
\usetikzlibrary{matrix,arrows,decorations.pathmorphing}
\usepackage{enumerate}
\usepackage{mathrsfs}
\usepackage{bbm}
\usepackage{color}
\newcommand{\N}{\mathbb{N}}
\newcommand{\R}{\mathbb{R}}
\newcommand{\Q}{\mathbb{Q}} 
\newcommand{\Z}{\mathbb{Z}}

\newcommand{\cob}{\textrm{cob}}
\newcommand{\diam}{\textrm{diam}}

\newcommand{\Union}{\bigcup}

\newcommand{\disUnion}{\bigsqcup}
\newcommand{\intersect}{\cap}
\newcommand{\Intersect}{\bigcap}
\newcommand{\Top}{{\rm top}}
\newcommand{\Base}{{\rm base}}

\newcommand{\Inf}{\textrm{Inf }}
\usepackage{draftwatermark}
\SetWatermarkScale{4}
\SetWatermarkLightness{.75}
\SetWatermarkText{}

\newcommand{\mchi}{\mathbbm{1}}

\theoremstyle{plain}
\newtheorem{main}{Main Theorem}
\newtheorem{theorem}{Theorem}[section]
\newtheorem{prop}[theorem]{Proposition} 
\newtheorem{lemma}[theorem]{Lemma}

\newtheorem{coro}[theorem]{Corollary}

\theoremstyle{definition}
\newtheorem{defn}[theorem]{Definition}
\newtheorem{remark}[theorem]{Remark}

\newtheorem{example}[theorem]{Example}
\newtheorem{note}[theorem]{Note}

\title[Speedups]{Topological Speedups for Minimal Cantor Systems}
\author[D.\ Ash]{Drew D. Ash}
\address{Department of Mathematics and Computer Science, Albion College, 611 E. Porter Street, Albion, MI 49224}
\email{dash@albion.edu}

\author[N.\ Ormes] {Nicholas S. Ormes}
\address{Department of Mathematics, University of Denver, 2280 S. Vine Street, Denver, CO 80208}
\email{normes@du.edu}

\keywords{Topological speedups, orbit equivalence, strong orbit equivalence, Kakutani-Rokhlin towers, dimension groups}
\subjclass[2010]{Primary 37B05; Secondary 37A20, 37A25, 37B10, 37B40}

\begin{document}
\maketitle
\begin{abstract}
	  In this paper we study speedups of dynamical systems in the topological category. Specifically, we characterize when one minimal homeomorphism on a Cantor space is the speedup of another. We go on to provide a characterization for strong speedups, i.e., when the jump function has at most one point of discontinuity. These results provide topological versions of the measure-theoretic results of Arnoux, Ornstein and Weiss, and are closely related to Giordano, Putnam and Skau's characterization of orbit equivalence for minimal Cantor systems. 
\end{abstract}

\section{Introduction}
Given a dynamical system $T:X \to X$ we generally define a \emph{speedup} of $(X,T)$ to be another dynamical system $S:Y \to Y$ conjugate to $T':X\rightarrow X$ where $T'(x)=T^{p(x)}(x)$ for some function $p:X\rightarrow\mathbb{Z}^{+}$. We will refer to $p$ in the preceding definition as the ``jump function.'' Speedups can be viewed in both the topological and measure-theoretic category. While this paper considers topological speedups, the study of speedups has a longer history in the measure-theoretic category. Of note are two papers of Neveu \cite{Neveu1, Neveu2}. The first paper proves a structure theorem relating jump functions for speedups to ``stopping times.'' His second paper extends Abramov's formula on the entropy of induced transformation to calculate the entropy of speedups with a measurable jump function. The most general theorem in this category came in a $1985$ paper of Arnoux, Ornstein, and Weiss, where they showed: for any ergodic measure preserving transformation $(X,\mathscr{B},\mu,T)$ and any aperiodic $(Y,\mathscr{C},\nu, S)$ there is a $\mathscr{B}-$measurable function $p:X\rightarrow\Z^{+}$ such that $T'(x)=T^{p(x)}(x)$ is invertible $\mu$-a.e. and $(X,\mathscr{B},\mu,T')$ is isomorphic to $(Y,\mathscr{C},\nu, S)$  \cite{AOW}. Finding a topological analogue to this theorem was the original impetus for this paper. 

While the result of Arnoux, Ornstein, and Weiss shows that almost no assumptions are needed to realize one measure-theoretic system as the speedup of another, there are restrictions that arise in the topological category. The situation closely mirrors that for orbit equivalence. Recall that two invertible dynamical systems $T:X \to X$ and $S:Y \to Y$ are \emph{orbit equivalent} if $(Y,S)$ is conjugate to $T':X \to X$ where $T$ and $T'$ have the same orbits. 
Dye's theorem says that any two ergodic transformations on non-atomic Lebesgue probability spaces are measure-theoretically orbit equivalent \cite{Dye}. Over 35 years later Giordano, Putnam, and Skau gave a characterization of when two minimal Cantor systems $(X,T)$ and $(Y,S)$ are topologically orbit equivalent \cite{GPS}. 

Here we briefly describe the dimension groups used in their characterization (see Section \ref{SectionDimGrp} for much more). The triple $\mathcal{G}(X,T)$ is defined as \linebreak $(G(X,T),G(X,T)_+,[\mchi_X])$ where $G(X,T)$ is the group of continuous integer-valued functions modulo functions of the form\linebreak $f-fT^{-1}$, $G(X,T)_+$ is the semi-group of equivalence classes of non-negative functions, and the distinguished order unit $[\mchi_X]$ is the equivalence class of the constant function 1 on $X$. The subgroup $\Inf(\mathcal{G}(X,T))$ is the set of functions which integrate to $0$ against every $T$-invariant Borel probability measure.

\begin{theorem}\textbf{\cite[Theorem $2.2$]{GPS}:}\label{GPS Theorem} Let $(X,T)$ and $(Y,S)$ be minimal Cantor systems. The following are equivalent:
	\begin{enumerate}
		\item $(X,T)$ and $(Y,S)$ are orbit equivalent.
		\item The dimension groups $\mathcal{G}(X,T)/\Inf(\mathcal{G}(X,T))$ and\linebreak $\mathcal{G}(Y,S)/\Inf(\mathcal{G}(Y,S))$, are order isomorphic by a map preserving the distinguished order unit.
		\item There exits a homeomorphism $F:X\rightarrow Y$ carrying the $T$-invariant probability measures onto the $S$-invariant probability measures.
	\end{enumerate}
\end{theorem}

In this paper, we show an analogous result but for speedups. Given unital ordered groups $\mathcal{G}$ and $\mathcal{H}$, we say that an epimorphism $\varphi:\mathcal{H} \twoheadrightarrow \mathcal{G}$ is \emph{exhaustive} if for every $h \in H_+$ and 
$g \in G$ satisfying $0 < g < \varphi(h)$, there is a $h' \in H$ such that $\varphi(h')=g$ and $0<h'<h$. 

\begin{main}\label{main} Let $(Y,S)$ and $(X,T)$ be minimal Cantor systems. The following are equivalent:
	\begin{enumerate}
		\item $(Y,S)$ is a speedup of $(X,T)$.
		\item There exists an exhaustive epimorphism\linebreak $\varphi:\mathcal{G}(Y,S)/\Inf(G(Y,S))\twoheadrightarrow\mathcal{G}(X,T)/\Inf(G(X,T))$ which preserves the distinguished order unit.
		\item There exists homeomorphism $F:X\rightarrow Y$, such that $F_{*}:M(X,T)\hookrightarrow M(Y,S)$ is an injection.
	\end{enumerate}
\end{main}

The unital ordered groups in both of these results arise from the study of $C^*$-crossed product algebras associated with the minimal Cantor systems. In particular, $(G(X,T),G(X,T)_+)$ is the ordered $K^0$-group for the crossed product algebra $C(X) \underset{T}{\ltimes} \Z$.
Giordano, Putnam and Skau provided a dynamical characterization of the isomorphism of these objects as well, via a relation called \emph{strong} orbit equivalence. Two minimal Cantor systems $(X,T)$ and $(Y,S)$ are strong orbit equivalent if there exists a homeomorphism $F:X \to Y$ such that
 $$FT(x) = S^{n(x)}F(x) \qquad \text{ and } \qquad FT^{m(x)}(x) = SF(x)$$
 where $m,n:X \to \Z$ have at most one point of discontinuity each.

\begin{theorem}{\bf \cite[Theorem ]{GPS}}\label{GPS Theorem Strong} Let $(X,T)$ and $(Y,S)$ be minimal Cantor systems. The following are equivalent:
	\begin{enumerate}
		\item $(X,T)$ and $(Y,S)$ are strongly orbit equivalent.
		\item The dimension group $\mathcal{G}(X,T)$ is order isomorphic to $\mathcal{G}(Y,S)$ by a map preserving distinguished order units.
		\item $C(X) \underset{T}{\ltimes} \Z$ is isomorphic to $C(Y)\underset{S}{\ltimes} \Z.$
	\end{enumerate}
\end{theorem}
In our second main theorem, we prove the following a speedup analog of Theorem~\ref{GPS Theorem Strong}. A \emph{strong speedup} is a speedup where the jump function has at most one point of discontinuity.
\begin{main}\label{mainstrong}
	Let $(Y,S)$ and $(X,T)$ be minimal Cantor systems. The following are equivalent:
	\begin{enumerate}
		\item $(Y,S)$ is a strong speedup of $(X,T)$.
		\item There is an exhaustive epimorphism $\varphi :\mathcal{G}(Y,S) \twoheadrightarrow \mathcal{G}(X,T)$ preserving the distinguished order unit.
		\item There is a homeomorphism $h:X \to Y$ such that if $f \in \cob(Y,S)$ then $fh \in \cob(X,T)$.
	\end{enumerate}
\end{main}

The paper is organized as follows. In Section 2 we provide relevant background on minimal Cantor systems, including Kakutani-Rohlin towers and dimension groups. In Section 3, we prove some basic facts about speedups and the jump function associated to them. In Section 4 we prove Main Theorem 2 where we characterize strong speedups, and in Section 5 we prove Main Theorem 1 for general speedups. The proofs for the two Main Theorems generally follow the same format. We prove Main Theorem 2 first as it requires fewer preliminary lemmas than Main Theorem 2. The proof of Main Theorem 1 follows the same format and we will highlight the differences and analogs statements. For the proofs, we follow many of the same techniques as in \cite{GPS} and \cite{Glasner-Weiss}, adapted to the speedup situation. 

In Section 6 we define and explore the relation of \emph{speedup equivalence}. That is, when two minimal Cantor systems are speedups of one another. It follows from our main theorem that if two minimal Cantor system are orbit equivalent then they are speedup equivalent. We show in Theorem \ref{speedupequivalence} that the converse is true when the systems are have finitely many ergodic invariant measures (in particular when they are uniquely ergodic). However, Julien Melleray's forthcoming paper \emph{Dynamical Simplicies and Fra\"iss\'e Theory} showed speedup equivalence and orbit equivalence are different relations in general \cite{Melleray}.
\section{Acknowledgments}
The authors thank the referee for their close reading of our paper and for their insightful suggestions. Additionally, the first author dedicates this paper to Laura Brade. Your unwavering belief in me and support of me made this paper possible. Thank you.
\section{Preliminaries}
\subsection{Minimal Cantor systems}
Throughout this paper $X$ will always be taken to be a Cantor space, that is a compact, metrizable, perfect, zero-dimensional space. A \emph{Cantor system} will consist of a pair $(X,T)$ where $X$ is a Cantor space and $T:X\rightarrow X$ is a homeomorphism. For any $x \in X$, we denote the \emph{orbit} of $x$ by 
$\mathcal{O}_{T}(x)=\{T^{n}(x):x\in\Z\}$ and the \emph{forward orbit} of $x$ by 
$\mathcal{O}^+_{T}(x)=\{T^{n}(x):x\in\Z^+\}$. We will restrict our attention in this paper to 
\emph{minimal Cantor systems}, i.e., Cantor systems $(X,T)$ in which every orbit is dense. 
It is well-known (see \cite{Walters}) that for a Cantor system $(X,T)$, the density of all orbits $\mathcal{O}_{T}(x)$ is equivalent to the density of all forward orbits $\mathcal{O}^+_{T}(x)$. 

Fix a minimal Cantor system $(X,T)$ and let $M(X)$ denote the collection of all Borel probability measures on $X$. We are interested in the set $M(X,T) \subset M(X)$ of those measures $\mu \in M(X)$ which are $T$-invariant, i.e., where
$\mu(T^{-1}(A))=\mu(A)$ for every Borel subset $A$. Because we are working with minimal Cantor systems, all $\mu \in M(X,T)$ have full support, and every $\mu \in M(X,T)$ is non-atomic.

The Bogolioubov-Krylov Theorem implies $M(X,T)\neq\emptyset$. Moreover, the set $M(X,T)$ is a compact, convex subset of $M(X)$. Further, the Ergodic Decomposition Theorem states every measure $\mu\in M(X,T)$ can be uniquely represented as an integral against a measure $\tau$ which is fully supported on the ergodic measures of $M(X,T)$.

\subsection{Kakutani-Rohlin towers}

We will consider minimal Cantor systems through a refining sequence of Kakutani-Rokhlin tower partitions. These tower partitions, defined below, were instrumental in relating minimal Cantor systems to Bratteli diagrams, and hence dimension groups, and $C^*$-algebras \cite{HPS, GPS,Glasner-Weiss}. We include proofs of some basic facts here for completeness. 
\begin{defn}
	A \emph{{Kakutani-Rokhlin tower partition}} (or \emph{KR-partition}) of a minimal Cantor system $(X,T)$ is a clopen partition $\mathcal{P}$ of $X$ of the form
	$$
	\mathcal{P}=\{T^{j}C(i): 1 \leq i\leq I,\, 0\le j<h(i)\}
	$$
	where $I \geq 1$, and for each $1 \leq i \leq I$, $C(i)$ is a clopen set and $h(i)$ is a positive integer. 
\end{defn}
Fixing $i \in \{1,\ldots ,I\}$, we refer to the collection of sets $\{T^{j}C(i):0\le j<h(i)\}$ as the $i^{th}$ \emph{column} of the partition
$\mathcal{P}$, and $h(i)$ as the height of this column. The set $T^{j}C(i)$ is the $j^{th}$ \emph{level} of the $i^{th}$ column. 
Furthermore, we refer to the union $\Union C(i)$ as the base of $\mathcal{P}$, 
$$
\Base (\mathcal{P}) =\Union_{1 \leq i\leq I}C(i)
$$
and the union $\Union T^{h(i)-1}C(i)$ as the top of $\mathcal{P}$,
$$\Top (\mathcal{P}) = \Union_{0 \leq i \leq I}T^{h(i)}C(i)
$$
as the \emph{top} of $\mathcal{P}$. 
Thus in the $i^{th}$ column of a KR-partition, if $0\leq j<h(i)$ then $T$ maps the $j^{th}$ level of the column onto the $(j+1)^{st}$ level of that same column. Moreover, $T$ maps $\Top (\mathcal{P})$ homeomorphically onto $\Base (\mathcal{P})$.

The following are three basic facts about the construction of KR-partitions which we will use extensively.

\begin{lemma} \label{returntime}
	Let $(X,T)$ be a minimal Cantor system and let $N$ be given. There is an $\varepsilon>0$ such that if $A$ is a clopen set with $\diam (A) < \varepsilon$ then for all $x \in A$, $r_A(x)>N$. 
\end{lemma}

\begin{proof}
	Suppose not. Then for all $k \geq 1$, there is a clopen set $A_k$ with $\diam(A_k)<\frac{1}{k}$ and $x_k$ with $1 \leq r_{A_k}(x_k) \leq N$. By passing to a subsequence, we may assume there is an $1 \leq r\leq N$ such that $r_{A_k}(x_k)=r$ for all $k \geq 1$. Thus we have $d(T^r(x_k),x_k)< \frac{1}{k}$ for all $k \geq 1$. Let $x$ be a limit point for the sequence $\{x_k\}$, the continuity of $T^r$ implies $T^r(x)=x$. But $T$ has no periodic points, a contradiction. 
\end{proof}

\begin{lemma}\label{simpletower}
	Let $(X,T)$ be a minimal Cantor system and let $C \subset X$ be a nonempty clopen subset of $X$. Let $\mathcal{Q}$ be any clopen partition of $X$. Then there is a clopen KR-partition $\mathcal{P}$ of $X$ which refines $\mathcal{Q}$ and where $\Base(\mathcal{P}) = C$. 
\end{lemma}
\begin{proof}
	For $x \in C$, define the first return time of $x$ to $C$ as follows $r(x) = \inf \{n>0 : T^n(x) \in C\}$. 
	Since $T$ is minimal, $r(x)$ is a well-defined positive integer for all $x \in X$ and the function 
	$r: C \to \mathbb{N}$ is uniformly bounded by an integer $N\geq 1$. 
	By the continuity of $T^n$ for each $n \geq 1$, each set $r^{-1}\{n\}$ is closed.  Therefore, 
	$C$ is partitioned into finitely many nonempty closed sets $r^{-1}\{n\}$, which therefore must be clopen. 
	
	For each $n \in \{1,\ldots ,N\}$, partition $r^{-1}\{n\}$ into clopen sets by considering the intersection of $r^{-1}\{n\}$ with elements of the partition $\bigvee_{j=0}^{n-1} T^{-j} \mathcal{Q}$. This results in a finite collection of nonempty, pairwise disjoint clopen sets $C(1),  \ldots , C(I)$ whose union is $C$. 
	
	Fix $i \in \{1,\ldots,I\}$. The set $C(i)$ is a  subset of  $r^{-1}\{n\}$ for some $n \geq 1$; set $h(i)=n$. 
	Since $C(i)$ is a subset of an element of  $\bigvee_{j=0}^{n-1} T^{-j} \mathcal{Q}$, if
	$0 \leq j < h(i)$, then $T^jC(i)$ is a subset of an element of $\mathcal{Q}$. Therefore, 
	$$\mathcal{P}=\{T^jC(i) : 1 \le i\le  I, 0 \leq j < h(i)\}$$ is the desired partition. 	
\end{proof}

\begin{lemma} \label{refinetower}
	Let $(X,T)$ be a minimal Cantor System and let 
	$\mathcal{P} = \{T^j C(i) : 1 \le i \le I, 0 \leq j < h(i)\}$ 
	be a KR-partition for $(X,T)$. Suppose $x_0 \in T^{h(1)-1}C(1)$
	and $T(x_0) \in C(2)$ and $\varepsilon>0$. Then there is a KR-partition 
	$\mathcal{Q} = \{T^j B(i) : 1 \le i \le I', 0 \leq j < k(i)\}$ such that 
	\begin{enumerate}
		\item $\mathcal{Q}$ refines $\mathcal{P}$,
		\item $\diam(T^j B(i))<\varepsilon$ for all $i,j$,
		\item $\Base(\mathcal{Q}) \subset C(2)$ and $\Top(\mathcal{Q}) \subset T^{h(1)-1}C(1)$
		\item $x_0 \in T^{h(1)-1}B(1)$ and $T(x_0) \in B(2)$. 
	\end{enumerate}
\end{lemma}
\begin{proof}
	Let $U, V$ be clopen sets such that $x_0 \in U \subset T^{h(1)-1}C(1)$ and 
	$T(x_0) \in V \subset C(2)$. Set $B = V \cap T(U)$. Then 
	$T(x_0) \in B \subset C(2)$ and $x_0 \in  T^{-1}B \subset T^{h(1)-1}C(1)$. 
	Let $r>0$ be the first return time of $T(x_0)$ to $B$. Since $T^r(x_0) \neq x_0$, 
	there is a clopen set $D \subset B$ containing $T(x_0)$ such that $x_0 \not \in T^{r-1}(D)$. 
	
	Let $\mathcal{R}$ be any partition with diameter less than $\varepsilon$. 
	Apply Lemma \ref{simpletower} to the set $B$ and the partition $\mathcal{P} \vee \{D,D^c\} \vee \mathcal{R}$ to obtain a KR-partition $\mathcal{Q}$ with base $B$. Properties (1)-(3) for $\mathcal{Q}$ follow. Because the $\mathcal{Q}$ refines $\{D,D^c\}$, the column that contains the point $T(x_0)$ does not also contain $x_0$. Therefore, after re-indexing the sets we can assume Property (4) as well. 
\end{proof}

\subsection{Dimension groups} \label{SectionDimGrp}
Like the Giordano, Putnam and Skau results, our main theorems are expressed in terms of the unital dimension group $\mathcal{G}(X,T)$. These partially ordered groups are examples of dimension groups, direct limits of ordered groups $(\Z^{n_k},(\Z_+)^{n_k})$. Here we will give an abstract definition of dimension groups, which as shown by Effros, Handelman and Shen is equivalent \cite{EHS}. The relevant definitions and facts are briefly summarized, see \cite{HPS},\cite{GPS},\cite{CANT} for much more. 

\begin{defn} A \emph{{partially ordered group}} is a countable Abelain group $G$ together with a special subset denoted $G_{+}$, referred to as the \emph{{positive cone}}, satisfying the following:
	\begin{enumerate}
		\item $G_{+}+G_{+}\subseteq G_{+}$
		\item $G_{+}-G_{+}=G$
		\item $G_{+}\intersect(-G_{+})=\{0\}$
	\end{enumerate}
\end{defn}
Given $a,b\in G$ we will write
$$
a\le b\text{ if } b-a\in G_{+} \qquad \text{ and } \qquad
a < b\text{ if } b-a\in G_{+}\setminus\{0\}.
$$
\begin{defn}
We will further require that our partially ordered Abelian groups be \emph{unperforated} by which we mean: if $a\in G$ and $na\in G_{+}$ for some $n\in\Z^{+}$ then $a\in G_{+}$.
\end{defn} 

\begin{defn}
	A partially ordered group is said to satisfy the \emph{{Riesz interpolation property}} if given $a_{1},a_{2},b_{1},b_{2}\in G$ with $a_{i}\le b_{j}$ for $i,j=1,2$, then there exists $c\in G$ such that 
	$$
	a_{i}\le c\le b_{j} \text{ for }i,j=1,2.
	$$
\end{defn}
\begin{note}
	 We will take the following as a definition, but it is really a theorem of Effros, Handelman and Shen, \cite{EHS}.
\end{note}
\begin{defn}
	A \emph{{dimension group}} is an unperforated partially ordered group $(G,G_{+})$ which satisfies the Riesz interpolation property.
\end{defn}
\begin{defn}
	An \emph{order unit} of a dimension group $(G,G_{+})$ is an element $u$ of $G_{+}$ with the property that for all positive $g\in G$ there exists a positive integer $N$ such that $g\le Nu$.
\end{defn}
\begin{defn}
	A dimension group is \emph{simple}  if every nonzero positive element is an order unit.
\end{defn}
\begin{defn}
	Given two dimension groups $(G,G_+)$ and $(H,H_+)$, and order units $u \in G$ and $v \in H$,  
	by a \emph{{unital dimension group isomorphism}} from $(G,G_+,u)$ to $(H,H_+,v)$, we mean a group isomorphism 
	$\varphi:G \to H$ such that 
	\begin{enumerate}
		\item $\varphi(G_+) = H_+$
		\item $\varphi(u) = v$.
	\end{enumerate}
\end{defn}

While unital ordered isomorphisms of $\mathcal{G}(X,T)$ and $\mathcal{G}(Y,S)$ play an important role in the characterization of minimal Cantor systems up to orbit equivalence, the relevant relation for us is the existence of an \emph{epimorphism} from one unital ordered group to another.  

\begin{defn}
	Given two dimension groups $(G,G_+)$ and $(H,H_+)$, and order units $u \in G$ and $v \in H$, 
	by a \emph{{unital dimension group epimorphism}} from $(G,G_+,u)$ onto $(H,H_+,v)$, we mean 
	a group epimorphism  $\varphi:G \to H$ such that 
	\begin{enumerate}
		\item $\varphi(G_+) \subseteq H_+$
		\item $\varphi(u) = v$.
	\end{enumerate}
	
\end{defn}

As it turns out, we will require one additional order property for the unital dimension group epimorphisms that arise in this paper. 

\begin{defn}
	Let $\mathcal{G}=(G,G_+,u)$ and $\mathcal{H}=(H,H_+,v)$ be unital ordered groups. 
	We say that a unital dimension group epimorphism from $\mathcal{G}$ to $\mathcal{H}$ is 
	\emph{exhaustive} if for every $g \in G_+$ and 
	$h \in H$ satisfying $0 < h < \varphi(g)$, there is a 
	$g' \in G$ such that $\varphi(g')=h$ and $0<g'<g$. 
\end{defn}

Note that every dimension group isomorphism is an exhaustive epimorphism. Further, a unital dimension group epimorphism having the exhaustive property implies $\varphi(G_+)=H_+$. 

We finish this subsection by providing several examples to illustrate the notion of an exhaustive epimorphism. The examples will be of the following forms. Any ordered group of these forms represent a simple dimension group. 

\begin{enumerate}
	\item $(G_1,(G_1)_+)$ where $G_1$ is a countable dense subgroup of $\mathbb{R}^2$, and $(G_1)_+ = \{ (x,y) \in G_1 : x>0 \text{ and } y>0\} \cup \{(0,0)\}$. 
	\item $(G_2,(G_2)_+)$ where $G_2$ is a countable dense subgroup of $\mathbb{R}^2$, and $(G_2)_+ = \{ (x,y) \in G_2 :  y>0\} \cup \{(0,0)\}$. 
	\item $(G_3,(G_3)_+)$ where $G_3$ is a countable dense subgroup of $\mathbb{R}$, and $(G_3)_+ = \{ x \in \mathbb{R} :  x \geq 0\}$. 
\end{enumerate}
If the example is of the type $(G_1,(G_1)_+)$ or $(G_3,(G_3)_+)$ above, we will say that the subset of $\R^2$ or $\R$ has the usual ordering. If it is of the type $(G_2,(G_2)_+)$ above, we will say that the ordering is given by the second coordinate. Because these are all simple, any positive element can play the role of the distinguished order unit.

\begin{example}\label{exh:ex1}
Consider the group $G = \Q \times \Q$ with the usual ordering and the map $\varphi: (G,G_+) \to (G,G_+)$ where $\varphi(x,y) = \left( \frac{2}{3} x + \frac{1}{3} y,\frac{1}{3} x + \frac{2}{3} y\right)$. This is a bijection from $G$ to itself with $\varphi(G_+) \subsetneq G_+$, and thus not exhaustive. 
\end{example}

\begin{example}\label{exh:ex2}
Consider the group $G_1 = \Q \times \Q$ with the usual ordering and $G_2 = \Q \times \Q$ with ordering given by the second coordinate. Consider the map $\varphi: (G_1,(G_1)_+) \to (G_2,(G_2)_+)$ where $\varphi(x,y) = \left( x-y,x+y\right)$. This is a bijection from $G_1$ to $G_2$, and $\varphi((G_1)_+) \subset (G_2)_+$. However, this is not exhaustive. For example, $(-1,1) \in (G_2)_+$, but $\varphi^{-1}(-1,1) = (0,1) \notin (G_1)_+$. 

On the other hand, let us consider these same ordered groups modulo their respective infinitesimal subgroups, which we denote as $(H_1,(H_1)_+)$ and $(H_2,(H_2)_+)$. Then $(G_1,(G_1)_+) = (H_1,(H_1)_+)$, but $(H_2,(H_2)_+)$ is isomorphic to $(\Q , \Q_+)$ as any element of the form $(x,0)$ is an infinitesimal in $(G_2,(G_2)_+)$. Then the induced map from $\varphi$ from $(H_1,(H_1)_+)$ to $(\Q,\Q_+)$, which we denote $\hat{\varphi}$, has the formula $\hat{\varphi}(x,y) = x+y$, and is an exhaustive epimorphism. So in particular, one cannot deduce the exhaustive property of a group epimorphism from the induced map on the groups modulo infinitesimals.
\end{example}

\begin{example}\label{exh:ex3}
Suppose $\alpha, \beta$ are irrational numbers that are not rationally related. 
Consider the group $G_1 = \{ (a \alpha + b, a \alpha + c \beta + d) \in \mathbb{R}^2 : a,b,c,d \in \Z \}$ 
with the usual ordering. Let $\varphi : G_1 \to \R$ be projection onto the second coordinate, let $G_3 = \varphi(G_1) \subset \R$ with the usual ordering. Note that in this case $\varphi ((G_1)_+) = (G_3)_+$. 

Now fix $c,d \in \Z$ such that $0 < c \beta + d < 1$. Let $h = c\beta +d \in G_3$, and $g = (1,1) \in G_1$. Then we have $0 < h < \varphi(g)$. However, $\varphi^{-1}\{h\} = \{(b,c \beta + d) : b \in \Z \}$. If $(b,c \beta + d) > (0,0)$ then $b>0$, but if $(b,c \beta + d) < (1,1)$ then $b<1$. Therefore, there is no element $g'$ in $G_1$ satisfying both $\varphi(g') = h$ and $0<g'<g$. 
\end{example}

\begin{remark}
Note that in Example \ref{exh:ex3} the kernel of $\varphi$ is $\{0\} \times \Z$. In particular, $\ker \varphi$ is a discrete subset of $\R^2$. For the setting where $G_1$ is a dense subset of $\R^2$ and $G_3$ is a dense subset of $\R$ each with the usual ordering, one can prove that a group epimorphism $\varphi:G_1\to G_3$ is exhaustive if and only if $\ker \varphi$ is a dense subset of a line in $\R^2$. There are certainly generalizations of this fact, but a complete characterization of when an ordered group epimorphism is exhaustive appears to be subtle and is beyond the scope of this paper. It could be that the key to such a characterization is to consider these ordered groups modulo infinitesimals as dense subsets of affine maps on compact sets as suggested by the referee of this paper. With this perspective, an epimorphism of ordered groups is the restriction of an affine map on a space to a subset.

\end{remark}

\subsection{Dynamical dimension groups}
For a minimal Cantor system $(X,T)$, consider  $C(X,\Z)$, the group of continuous integer-valued functions on $X$. Let $\cob(X,T)$ denote the subgroup of \emph{$T$-coboundaries}, i.e. $\cob(X,T)=\{f-f\circ T:f\in C(X,\Z)\}$.
Furthermore, define
$$
G(X,T)=C(X,\Z)/\cob(X,T).
$$ 
Next we define the \emph{positive cone} $G(X,T)_+ \subset G(X,T)$ to be
$$
G(X,T)_{+}=\{[f]: f\in C(X,\Z), f(X) \subset \Z_+\}.
$$ 
and the distinguished order unit $[\mchi_X]$ to be the equivalence class of the constant function one. 
The triple $\mathcal{G}(X,T) = (G(X,T),G(X,T)_+,[\mchi_X])$ forms a unital dimension group and further,
every simple, unital dimension group except $(\Z,\Z_+,1)$ arises as $(G(X,T),G(X,T)_+,[\mchi_X])$ for some minimal Cantor system $(X,T)$ as shown by Giordano, Putnam, and Skau, \cite{GPS}.

The second unital dimension group we will consider is a quotient of $\mathcal{G}(X,T)$ above. 
Set $$\Inf(X,T) = \left\{f \in C(X,\Z) : \int f d \mu = 0 \text{ for all } \mu \in M(X,T) \right\}$$
and notice $ \cob(X,T)\subset \Inf(X,T)$. We obtain another unital dimension group by considering the group
$G(X,T)/\Inf(X,T)$, the positive cone given by the equivalence classes of nonnegative functions, and the 
distinguished order unit given by the equivalence class of the constant function one. We will denote this 
unital dimension group by $\mathcal{G}(X,T)/\Inf(X,T)$.

\section{Speedup Basics}

Here we consider some basic properties of a speedup of a minimal Cantor system. 

\begin{defn}
Given minimal Cantor systems $(X,T)$ and $(Y,S)$, we say that $(Y,S)$ is a \emph{speedup} of $(X,T)$ if there is a function $p:X\rightarrow\mathbb{Z}^{+}$ such that $(Y,S)$ is topologically conjugate to $(X,T')$ where $T'(x)=T^{p(x)}(x)$ for all $x \in X$.
\end{defn}

We will refer to the function $p$ above a jump function for the speedup. While we place no continuity conditions on the jump function in the definition above, the assumption that $(X,T')$ is conjugate to a minimal Cantor system $(Y,S)$ implies that $(X,T')$ is also a minimal Cantor system, which implies that $p$ is lower semicontinuous as seen below. 

\begin{prop} Suppose $(X,T)$ is a minimal Cantor system. Let $p:X\rightarrow\mathbb{Z}^{+}$ and suppose that $T'(x)=T^{p(x)}(x)$ is continuous. Then $p$ is lower semicontinuous.
\end{prop}
\begin{proof}
	It suffices to show that for every $n\in\Z^{+}$ $p^{-1}(\{n\})$ is closed. The lower semicontinuity of $p$ will then follow since for any $\alpha \in \R$, $\{x:p(x)>\alpha\}$ is the complement of union of finitely many sets of the form $p^{-1}(\{n\})$. 
	
	To this end, 
	let $n\in\mathbb{Z}^{+}$, $\{x_{m}\}_{m\ge 1}\subseteq p^{-1}(\{n\})$ and $x\in X$ such that $x_{m}\rightarrow x$. Since both $T'$ and $T^{n}$ are continuous, we have
	$$
	T'(x_{m})\rightarrow T'(x)\text{ and } T^{n}(x_{m})\rightarrow T^{n}(x).
	$$
	Since for every $m$, $T'(x_{m})=T^{n}(x_{m})$ we have that
	$$
	T'(x)=T^{n}(x).
	$$
	We may conclude $p(x)=n$ as a result of $T$ being aperiodic.
	
\end{proof}

It follows from the above that if $p$ is bounded then each of the sets $p^{-1}(\{n\})$ is open as well. Thus we have the following. 

\begin{coro}
	Let $(X,T)$ be a minimal Cantor system. Suppose $(X,T')$ is a minimal Cantor system, where $p:X\rightarrow\mathbb{Z}^{+}$, and $T'(x)=T^{p(x)}(x)$ for all $x \in X$. 
	Then $p$ is continuous if and only if $p$ is bounded. 
\end{coro}

In the case where $(Y,S)$ is a speedup of $T$ with a bounded jump function, we call $S$ a \emph{bounded speedup} of $T$. Bounded speedups can be viewed as a topological analogue of the speedups studied by Neveu, where the jump function is integrable. Bounded speedups are interesting in their own right and are the subject of the paper \cite{AAO}. 

As in the case for orbit equivalence, strong topological speedups, where the jump function has a single point of discontinuity deserve special attention. 

\begin{defn}
Given a minimal Cantor systems $(X,T)$ and $(Y,S)$, we say that $(Y,S)$ is a \emph{strong speedup} of $(X,T)$ if there is a function $p:X\rightarrow\mathbb{Z}^{+}$ with at most one point of discontinuity in $X$ such that $(Y,S)$ is topologically conjugate to $(X,T')$ where $T'(x)=T^{p(x)}(x)$ for all $x \in X$.
\end{defn}

\section{Proof of Main Theorem \ref{mainstrong}}
In this section we prove Main Theorem \ref{mainstrong} which is our characterization theorem for strong speedups. We will prove the various implications as separate theorems in the subsequent three subsections.  

\subsection{Proof of (1) $\implies$ (3)}

In Section 3 we defined  $\cob(X,T)=\{f-f\circ T:f\in C(X,\Z)\}$. Note that by replacing $f$ with $g=-fT$ in $f-f\circ T$, we obtain $f-f\circ T = g- g \circ T^{-1}$, so 
$\cob(X,T)$ can be written as $\{g-g\circ T^{-1}:g\in C(X,\Z)\}$.

\begin{theorem}
Let $(Y,S)$ and $(X,T)$ be minimal Cantor systems. If $(Y,S)$ is a strong speedup of $(X,T)$ then there is a homeomorphism $h:X \to Y$ such that if $f \in \cob(Y,S)$ then $fh \in \cob(X,T)$.
\end{theorem}
\begin{proof}
	Suppose $(X,T)$ and $(Y,S)$ are minimal Cantor systems and $(Y,S)$ is conjugate to $(X,T')$ where $T'(x) = T^{p(x)}(x)$ for all $x \in X$ and $p:X \to \Z$ has a single point of discontinuity $x_0$. 
	
	Let $h:X \to Y$ denote the conjugacy between $(Y,S)$ and $(X,T')$. Then $h$ is a homeomorphism. It remains to show that if $f \in \cob(Y,S)$ then $fh \in \cob(X,T)$. 
	
	Suppose $f \in \cob(Y,S)$. Then $f=g-gS^{-1}$ for $g \in C(Y,\Z)$. Note that for any constant $k$, if $g'=g+k$ then $f=g'-g'S^{-1}$. Therefore, we may assume without loss of generality that $g(h(x_0))=0$. 
	
	Since $g \in C(Y,\Z)$, we have $g=\sum_{i=1}^I n_i \mathbbm{1}_{B_i}$ for some nonzero integers $n_i$ and clopen sets $B_i \subset Y$. Since $g(h(x_0))=0$, we may assume $h(x_0) \not \in \cup_{i=1}^I B_i$. Set $A_i = h^{-1}B_i$ and $A_i(k) = A_i \cap p^{-1}\{k\}$. Since $p$ is continuous away from $x_0$, we have that each $A_i(k)$ is clopen and there are only finitely many sets $A_i(k)$. Compiling all of this information, we see 
	\begin{align*}
	gh-gS^{-1}h & = \sum_i n_i\mathbbm{1}_{B_i}h -n_i\mathbbm{1}_{SB_i}h \\ 
	& = \sum_i n_i\mathbbm{1}_{A_i} -n_i\mathbbm{1}_{h^{-1}SB_i} \\
	& = \sum_{i,k} k\mathbbm{1}_{A_i(k)} -k\mathbbm{1}_{T^k A_i(k)}  \\
	& = \sum_{i,k} k\mathbbm{1}_{A_i(k)} -k\mathbbm{1}_{A_i(k)} T^{-k}  \in \cob(X,T).
	\end{align*}
\end{proof}

\subsection{Proof of (3) $\implies$ (2)}

We begin with some preliminary lemmas.

\begin{lemma} \label{towercob}
	Let $f \in C(X,\mathbb{Z})$. Then $f \in \cob(X,T)$ if and only if there is an $\varepsilon>0$ such that for any clopen set $A$,
	if $\diam(A) < \varepsilon$ then for all $x \in A$, $$ \sum_{j=0}^{r_A(x)-1} fT^j(x) = 0.$$
\end{lemma}

\begin{proof}
	Suppose $f \in  \cob(X,T)$. Then $f = gT-g$ where $g$ is uniformly continuous. Therefore there is an $\varepsilon>0$ such that if $d(x,y)<\varepsilon$, $g(x)=g(y)$. Let $A$ be any clopen set with $\diam(A)< \varepsilon$. 
	Let $r_A : A \to \N$ denote the first return time function to $A$. 
	Then for any $x \in A$ we have $gT^{r_A(x)}(x)-g(x) =0$ and thus
	$$
	\sum_{j=0}^{r_A(x)-1} fT^j(x)  =  \sum_{j=0}^{r_A(x)-1} (gT - g )T^j(x)= gT^{r_A(x)}(x)-g(x) =0
	$$
	
	Conversely, suppose there is a clopen set $A$ such that for all $x \in A$,
	$ \sum_{j=0}^{r_A(x)-1} fT^j(x) = 0$. Fix $a \in A$. Set $g(a)=0$, and 
	for $1 \leq j < r_A(a)$, set $$gT^j(a) = \sum_{k=0}^{j-1}fT^k(a)$$
	This defines $g$ on all of $X$ since every point in $X$ is equal to 
	$T^j(a)$ for some $a \in A$ and $0 \leq j < r_A(a)$. 
	
	If $x \in A$ then $$gT(x)-g(x) = f(x).$$ 
	Now suppose $x = T^j(a)$ where $0<j<r_A(x)-1$. Then 
	$$gT(x)-g(x) =  \sum_{k=0}^{j}fT^k(a) -  \sum_{k=0}^{j-1}fT^k(a) = fT^j(a) = f(x).$$
	Finally, assume $x =  T^{r_A(a)-1}(a)$. Then $T(x) \in A$, so 
	$$gT(x)-g(x) = -  \sum_{k=0}^{r_A(a)-2}fT^k(a).$$
	But since $ \sum_{k=0}^{r_A(a)-1}fT^k(a) =0$, we have 
	$$ -  \sum_{k=0}^{r_A(a)-2}fT^k(a) =fT^{r_A(a)-1}(a) = f(x). $$
	Thus $gT-g=f$ on all of $X$. 
\end{proof}

\begin{lemma} \label{towerpos}
Let $(X,T)$ be a minimal Cantor system and suppose $[f] \in G(X,T)$. Then $[f] >0$ if and only if there is an $\varepsilon>0$ such that for any clopen set $A$, if $\diam(A) < \varepsilon$ then for all $x \in A$, $$ \sum_{j=0}^{r_A(x)-1} fT^j(x) > 0.$$
\end{lemma}
\begin{proof}
Suppose $[f] >0$. Then $f$ is equal to a nonzero, nonnegative function $g$ plus an element $k$ of $\cob(X,T)$.
There is a nonempty clopen set $C$ such that for $x \in C$, $g(x)>0$. Since $T$ is minimal, there is an $r>1$ such that for all $x \in X$, $\{x,T(x), \ldots T^{r-1}(x) \} \cap C \neq \emptyset$. Therefore, if $n > 2r$, we have 
$$\sum_{j=0}^{n-1} gT^j(x) \geq \frac{1}{r+1} n>0.$$ 

By Lemmas \ref{towercob} and \ref{returntime}, there is an $\varepsilon>0$ such that for any clopen set $A$, if $\diam(A) < \varepsilon$ then for all $x \in A$, $r_A(x) > 2r$ and 
 $$ \sum_{j=0}^{r_A(x)-1} kT^j(x) = 0. $$ Therefore, we have
$$ \sum_{j=0}^{r_A(x)-1} fT^j(x) = \sum_{j=0}^{r_A(x)-1} gT^j(x)  > 0.$$

Conversely, suppose there is a clopen set $A$ such that for all $x \in A$, $$ \sum_{j=0}^{r_A(x)-1} fT^j(x) > 0.$$ 
For $x \in A$, set
$$f'(x) = \sum_{j=0}^{r_A(x)-1} fT^j(x)$$ otherwise set $f'(x)=0$. 
Then $f'(x) \geq 0$ for all $x \in X$ and $f'$ is a nonzero function. 
Further, we have 
$$ \sum_{j=0}^{r_A(x)-1} fT^j(x) = \sum_{j=0}^{r_A(x)-1} f'T^j(x)$$
which by Lemma \ref{towercob} implies $[f]=[f']>0$. 
\end{proof}

\begin{lemma} \label{functionineq}
	Let $(X,T)$ be a minimal Cantor system. Suppose $[f], [g] \in G(X,T)$ such that $0<[f]<[g]$. Then there is an $f'\in C(X,\Z)$ such that $0 \leq f'(x) \leq g(x)$ for all $x \in X$ and $[f']=[f]$.
\end{lemma}
\begin{proof}
	Both $[f]$ and $[g-f]$ are in $G(X,T)_+$. Take $\varepsilon$ to be smaller than $\varepsilon_f$ and  $\varepsilon_{g-f}$ from Lemma \ref{towerpos}. Use Lemma \ref{refinetower} to create a KR-partition 
	$\mathcal{P} = \{T^jA(i) : 1 \leq i \leq I, 0 \leq j < h(i)\}$ over a set $A$ with diameter less than $\varepsilon$ such that $g$ and $f$ are constant on every level of $\mathcal{P}$. 
	Then for any $x \in A(i)$, we have 
	$$\sum_{j=0}^{h(i)-1} gT^j(x) \geq \sum_{j=0}^{h(i)-1} fT^j(x) \geq 0.$$
	Therefore, it is possible to select values for a new function $f'$ on levels of the $i$th column of $\mathcal{P}$ so that $g(x) \geq f'(x) \geq 0$ for all $x$ and so that $$\sum_{j=0}^{h(i)-1} f'T^j(x) = \sum_{j=0}^{h(i)-1} fT^j(x).$$
	By Lemma \ref{towercob}, $[f']=[f]$. 
\end{proof}
By setting $g$ equal to the indicator function of a clopen set in Lemma \ref{functionineq}, we immediately obtain the following. 

\begin{coro} \label{clopensubset}
	Let $(X,T)$ be a minimal Cantor system and $A \subset X$ a clopen set. Suppose $[f] \in G$ such that $0<[f]<[\mathbbm{1}_A]$. Then there is a clopen set $C \subset A$ such that $[\mathbbm{1}_C]=[f]$.
\end{coro}

We are now ready to prove (3) $\implies$ (2) in Main Theorem \ref{mainstrong}. 
\begin{theorem}
Let $(X,T)$ and $(Y,S)$ be minimal Cantor systems. Suppose there is a homeomorphism $h:X \to Y$ such that if $f \in \cob(Y,S)$ then $fh \in \cob(X,T)$. Then there is an exhaustive epimorphism $\varphi :\mathcal{G}(Y,S) \twoheadrightarrow \mathcal{G}(X,T)$ preserving the distinguished order unit. 
\end{theorem}
\begin{proof}
	Suppose there is a homeomorphism $h:X\to Y$ such that if $f \in \cob(Y,S)$ then $fh \in \cob(X,T)$. Define $\varphi:G(Y,S) \to G(X,T)$ via $\varphi:[f] \mapsto [fh]$. Then $\varphi$ is a homomorphism, $\varphi(G(Y,S)_+)=G(X,T)_+$, and $\varphi[\mathbbm{1}_Y]=[\mathbbm{1}_X]$. Moreover, $\varphi$ is onto since $\varphi[fh^{-1}]=[f]$.
	
	To see the exhaustive property, suppose that $0<[f]<\varphi[g]$ where $[g]\in G(Y,S)_+$ and $[f]\in G(X,T)$. Then $0<[f]<[gh]$ so by Lemma \ref{functionineq}, there is an $f' \in C(X,\Z)$ such that $0 \leq f'(x) \leq gh(x)$ for all $x \in X$ and $[f']=[f]$. Now consider $[f'h^{-1}] \in G(Y,S)$. Since $f'h^{-1}(y) \geq 0$ for all $y \in Y$ we have $[f'h^{-1}]\in G(Y,S)_+$ and $\varphi[f'h^{-1}]=[f']=[f]<[gh]$. Since $[f']=[f] \neq [0]$, $f'$ is a nonzero function, thus $0<[f]<[gh]$ as desired. 
\end{proof}

\subsection{Proof of (2) $\implies$ (1)}

This is the most complicated of the implications, and thus we will need multiple preliminary lemmas. A central theme of these lemmas is to convert relations between unital dimension groups to relations between sets. 

\begin{lemma}\label{clopenpartition}
	Let $(X,T)$ and $(Y,S)$ be minimal Cantor systems. Assume there is an exhaustive epimorphism $\varphi :\mathcal{G}(Y,S) \twoheadrightarrow \mathcal{G}(X,T)$. Let $A \subset X$ and $B \subset Y$ be clopen sets such that $\varphi[\mathbbm{1}_B]=[\mathbbm{1}_A]$. Then for any clopen partition $\{A_1,A_2,\ldots,A_n\}$ of $A$ there is a clopen partition $\{B_1,B_2,\ldots,B_n\}$ of $B$ such that $\varphi[\mathbbm{1}_{B_i}]=[\mathbbm{1}_{A_i}]$ for all $1 \leq i \leq n$.
	
	Moreover, given any $n$ distinct 
	points $y_1,y_2,\ldots ,y_n \in B$ 
	we may select $B_i$ to contain $y_i$ for $i=1,2,\ldots,n$.  
\end{lemma}

\begin{proof}
	We work recursively, set $A_0=\emptyset$ and $B_0=\emptyset$.
	
	Now suppose $1 \leq k \leq n$ and we have defined disjoint clopen sets $B_0,B_1,\ldots B_{k-1}$ such that $B_i\subset B$ and $\varphi[\mathbbm{1}_{B_i}]=[\mathbbm{1}_{A_i}]$ for $0\leq i < k$. There are two cases: Either $k<n$ or $k=n$. 
	
	If $k<n$, then $$0 < [\mathbbm{1}_{A_k}] < [\mathbbm{1}_A]-\sum_{i=0}^{k-1} [\mathbbm{1}_{A_i}]
	=\varphi[\mathbbm{1}_B]-\sum_{i=0}^{k-1} \varphi [\mathbbm{1}_{B_i}].$$
	Since $\varphi$ is exhaustive, there is an $[f]\in G(Y,S)$ such that $$0<[f]<[\mathbbm{1}_B]-\sum_{i=0}^{k-1} [\mathbbm{1}_{B_i}]$$ and $\varphi[f]=[\mathbbm{1}_{A_k}]$. By Corollary \ref{clopensubset}, there is a clopen set $B_k \subset B\setminus \cup_{i=0}^{k-1} B_i$ such that $[f]=[\mathbbm{1}_{B_k}]$ so $\varphi[\mathbbm{1}_{B_k}]=[\mathbbm{1}_{A_k}]$.
	
	For $k=n$, set $B_n=B\setminus \cup_{i=0}^{n-1} B_i$. It follows that 
	$$\varphi[\mathbbm{1}_{B_n}]=\varphi[\mathbbm{1}_B]-\sum_{i=0}^{n-1} \varphi [\mathbbm{1}_{B_i}]
	=[\mathbbm{1}_A]-\sum_{i=0}^{n-1} [\mathbbm{1}_{A_i}]
	=[\mathbbm{1}_{A_n}].$$
	
	For the ``Moreover...'' part, assume that $y_1 \in B_i$. By minimality there is a $k \in \mathbb{Z}$ such that $S^k(y_1) \in B_1$ (in fact there are infinitely many).
	Furthermore, there is a clopen set $U \subset B_i$ containing $y_1$ such that $S^k(U) \subset B_1$. Set $B_1' = (B_1\setminus S^kU) \cup U$, and set 
	$B_i' = (B_i\setminus U) \cup S^k(U)$. 	Then $y_1 \in B_1'$ and it is not difficult to see that $[\mathbbm{1}_{B_1}]=[\mathbbm{1}_{B_1'}]$ and $[\mathbbm{1}_{B_i}]=[\mathbbm{1}_{B_i'}]$. 
	
	To arrange $y_i \in B_i$ for $i>1$, we apply the same idea. We simply need to take care for these points that we choose $U$ and $k \in \mathbb{Z}$ so that $U \cap S^k(U)$ does not contain $y_j$ for $j<i$. 
\end{proof}

\begin{lemma} \label{clopenpartition2}
	Let $(X,T)$ and $(Y,S)$ be minimal Cantor systems. Assume there is an epimorphism $\varphi :\mathcal{G}(Y,S) \twoheadrightarrow \mathcal{G}(X,T)$. Let $A \subset X$ and $B \subset Y$ be clopen sets such that $\varphi[\mathbbm{1}_B]=[\mathbbm{1}_A]$. Then for any clopen partition $\{B_1,B_2,\ldots,B_n\}$ of $B$ there is a clopen partition $\{A_1,A_2,\ldots,A_n\}$ of $A$ such that $\varphi[\mathbbm{1}_{B_i}]=[\mathbbm{1}_{A_i}]$ for all $1 \leq i \leq n$.
	
	Moreover, given any $n$ distinct points $x_1,x_2,\ldots ,x_n \in A$ we may select $A_i$ to contain $x_i$ for $i=1,2,\ldots,n$.  
	
\end{lemma}

The proof of Lemma \ref{clopenpartition2} is nearly identical to Lemma \ref{clopenpartition} except we do not need the exhaustive property, only that $\varphi(G(Y,S)_+)=G(X,T)_+$. 
To say a bit more, note that if $0 < [\mathbbm{1}_{B_i}] < [\mathbbm{1}_B]$ then it follows that $0<\varphi [\mathbbm{1}_{B_i}] < [\mathbbm{1}_A]$ and thus there is a clopen set $A_i \subset A$ with $[\mathbbm{1}_{A_i}]=\varphi [\mathbbm{1}_{B_i}]$ by Corollary \ref{clopensubset}.

Given an exhaustive epimorphism from one unital dimension group to another
$\varphi:\mathcal{G}(Y,S) \to \mathcal{G}(X,T)$, Lemma \ref{clopenpartition} above will allow us to take a KR-partition for $(Y,S)$ and mirror that partition in $X$ in a way that respects $\varphi$. While the partition in $X$ need not be a KR-partition, it will be KR-like in the following sense. 

\begin{defn}
	Let $(X,T)$ be a minimal Cantor system. We say a clopen partition 
	$\{A(i,j): 1 \leq i \leq I, 0\leq j < h(i) \}$ is \emph{KR-like} provided for each $i$ and $0\leq j,j'<h(i)$, $[\mathbbm{1}_{A(i,j)}]=[\mathbbm{1}_{A(i,j')}]$.
\end{defn}

The following indicates how we will use the above lemmas to mirror a KR-partition for one minimal Cantor system $(Y,S)$ in a KR-like partition for another minimal Cantor system $(X,T)$.

\begin{lemma}\label{copytower}
	Suppose there is an exhaustive epimorphism $\varphi: \mathcal{G}(Y,S) \twoheadrightarrow \mathcal{G}(X,T)$. Fix $x_0 \in X$. 
	Let $\mathcal{Q}=\{S^jB(i) : 1 \leq i \leq I, 0 \leq j < h(i)\}$ be a clopen KR-partition of $Y$ with $y_0 \in S^{h(1)-1}B_1$ and
	$S(y_0) \in B_2$. 
	Then there is a KR-like partition $\mathcal{P}  = \{A(i,j): 1 \leq i \leq I, 0 \leq j < h(i)\}$ of $X$ such that:
	\begin{enumerate}
		\item $x_0 \in A(1,h(1)-1)$ and $T(x_0) \in A(2,0)$.
		\item For every $(i,j)$, $\varphi([\mathbbm{1}_{S^jB_i}])=[\mathbbm{1}_{A(i,j)}]$.
	\end{enumerate}  
\end{lemma}
\begin{proof}
	Since $\varphi$ is unital, we have $\varphi[\mathbbm{1}_{Y}]=[\mathbbm{1}_{X}]$. Hence, by Lemma~\ref{clopenpartition} there exists clopen partition $\mathcal{P}=\{A(i,j): 1 \leq i \leq I, 0 \leq j < h(i)\}$ such that for each $(i,j)$
	$$
	\varphi([\mathbbm{1}_{S^jB(i)}])=[\mathbbm{1}_{A(i,j)}].
	$$ 
	Notice that the KR-like property of $\mathcal{P}$ follows from the above relation.	 
	Additionally, Lemma~\ref{clopenpartition} ensures that $x_0\in A(1,h(1)-1)$ and $T(x_0)\in A(2,0)$.
\end{proof}

Once a KR-like partition in $X$ is created, we may use the following lemma to begin to define a speedup $T'$ of $T$. 

\begin{lemma}[Strong Speedup Lemma]\label{speedup}
	Let $(X,T)$ be a minimal Cantor system and let $A,C$ be nonempty disjoint clopen subsets of $X$. If $\mathbbm{1}_{A}-\mathbbm{1}_{C}\in\cob(X,T)$, then there exists a continuous $p:A\rightarrow \Z^{+}$ such that $T'$, defined by $T'(x)= T^{p(x)}(x)$, is a homeomorphism from $A$ to $C$.
\end{lemma}
\begin{proof}
	Let $A$ and $C$ be as above. Since $\mathbbm{1}_{A}-\mathbbm{1}_{C}\in\cob(X,T)$, there exists a KR-partition $\mathcal{P}=\{T^{j}D(i):1\le i\le n,\, 0\le j<h(i)\}$, such that
	$\mathbbm{1}_{A}$ and $\mathbbm{1}_{C}$ are constant on all levels of $\mathcal{P}$ and $\mathbbm{1}_{A}-\mathbbm{1}_{C}$ sums to zero on all columns of $\mathcal{P}$. 
	
	Fix $i$ and consider the $i^{th}$ column of $\mathcal{P}$. By the above construction, there exists a bijection
	$$\Gamma: \{j : 0 \leq j < h(i), T^jD(i) \subset A\}
	\to \{j : 0 \leq j < h(i), T^jD(i) \subset C\}.$$
	For $x \in T^jD(i) \subset A$, set 
	$$p(x) = \min\{n > 0 : T^n(x) \in T^{\Gamma(j)}D(i) \}.$$
	That $p(x)$ is well-defined follows from the fact that the forward orbit of $x$ is dense. That $p$ is continuous on $T^jD(i)$ follows from the fact that both $T^jD(i)$ and $T^{\Gamma(j)}D(i)$ are clopen. 
	
	If $\Gamma(j)>j$ then $p$ is simply the constant $\Gamma(j)-j$ on $T^jD(i)$, and
	$T^{\Gamma(j)-j}$ yields a homeomorphism from $T^jD(i)$ to  $T^{\Gamma(j)}D(i)$.
	
	Now suppose $\Gamma(j)<j$. Then $p(x) = r(x) - (j-\Gamma(j))$ where $r(x)$ is the first return time of $x$ to $T^jD(i)$. It follows from the fact that the first return map $x \mapsto T^{r(x)}(x)$ is a homeomorphism from $T^jD(i)$ to itself that $x \mapsto T^{p(x)}(x)$ is  a homeomorphism from $T^jD(i)$ to $T^{\Gamma(j)}D(i)$. 
\end{proof}

With these basic lemmas established, we need to show how to consistently define $T'$ on a refining sequence of KR-like partitions of $X$, thus defining $T'$ on more and more of the space.  

\begin{lemma}[Refinement of mirrored partitions]\label{refinemirror}
	Suppose $(X,T)$ is a minimal Cantor system with KR-like clopen partition $\mathcal{P} = \{A(i,j): 1 \leq i \leq I, 0 \leq j < h(i) \}$ and $\varepsilon>0$, then there exists another KR-like clopen partition $\mathcal{P}' =\{A'(k,j): 1 \leq k \leq K, 0 \leq j < h(i)\}$ with the following properties:
	\begin{enumerate}
		\item For all $k,j$, $\diam(A'(k,j))<\varepsilon$.
		\item $\mathcal{P}'$ refines $\mathcal{P}$.
	\end{enumerate}
\end{lemma}
\begin{proof}
Fix $i \in \{1,\ldots ,I\}$ and $0 \leq j <h(i)$. 
Partition $A(i,j)$ into finitely many sets, $A(i,1,j), \ldots , A(i,L,j)$ each of diameter less than $\varepsilon$. Now for $j' \neq j$ and $0\leq j' < h(i)$, we have $\varphi [\mathbbm{1}_{A(i,j)}]=[\mathbbm{1}_{A(i,j')}]$. So, applying  Lemma~\ref{clopenpartition2} to the identity map there exist clopen sets 
$A(i,1,j'), \ldots , A(i,L,j')$ such that
$$
\varphi [\mathbbm{1}_{A'(i, \ell ,j')}]=[\mathbbm{1}_{A'(i, \ell ,j)}]
$$
for $1 \leq \ell \leq L$. 

Continue in the same manner for every pair $(i,j)$. After finitely many iterations and re-indexing, the result is a KR-like clopen partition $\mathcal{P}'=\{A'(k,j): 1 \leq k \leq K, 0 \leq j < h(i)\}$ which refines $\mathcal{P}$ and satisfies $\diam (A'(k,j)) < \varepsilon$ for all $k,j$. 
\end{proof}

\begin{lemma}\label{refineKR}
	Suppose there is an exhaustive epimorphism\linebreak $\varphi:\mathcal{G}(Y,S)\twoheadrightarrow \mathcal{G}(X,T)$. If $\mathcal{Q}$ is a KR-partition of $Y$ and $\mathcal{P}$ is a ``mirrored'' clopen partition of $X$ ($\mathcal{P}$ satisfies Lemma~\ref{copytower}), then any tower-preserving refinement of $\mathcal{Q}$, and can be mirrored onto $\mathcal{P}$ while still preserve Lemma~\ref{copytower}.
\end{lemma}
\begin{proof}
	Let $\mathcal{Q}=\{S^jB_i:1\le i\le I,0\le j<h(i)\}$ be a KR-partition of $Y$ and $\mathcal{P}=\{A(i,j)\}$ be as in Lemma~\ref{copytower}. Suppose $\mathscr{P}=\{P_n\}_{n=1}^{m}$ is a finite, clopen partition of $Y$. Consider $P_1$ and assume for some fixed $i,j$ that
	$$
	P_1\cap S^jB(i)\neq \emptyset \text{ and } S^j_2B(i)\nsubseteq P_1.
	$$
	In this case we split $S^j_2B(i)$ into two clopen sets:
	$$
	P_1\cap S^jB(i) \text{ and } P_1^c\cap S^jB(i).
	$$
	We now split the entire tower $\{S^{\ell}B(i)\}$ into the following two towers:
	$$
	\{S^{k-j}(P_1\cap T^j_2B(i))\}_{k=0}^{h(i)-1}\text{ and } \{S^{k-j}(P_1^c\cap S^jB(i))\}_{k=0}^{h(i)-1}.
	$$
	Observe, for each $k$ we have,
	$$S^kB(i)=S^{k-j}(P_1\cap S^j_2B(i))\cup S^{k-j}(P_1^c\cap S^jB(i)),$$
	thus by Lemma~\ref{clopenpartition2} there exist disjoint clopen subsets $A_1(i,k),A_c(i,k)$ such that 
	$$
	\varphi[\mathbbm{1}_{S^{k-j}(P_1\cap S^jB(i))}]=[\mathbbm{1}_{A_1(i,k)}]\text{ and } \varphi[\mathbbm{1}_{S^{k-j}(P_1^c\cap S^jB(i))}]=[\mathbbm{1}_{A_c(i,k)}].
	$$ 
	Continuing in this same manner for each $P_n\in\mathscr{P}$ and every pair $(i,j).$ After finitely many steps we both can refine $\mathcal{Q}$ to $\mathcal{Q}'$ so that $\mathcal{Q}'=\{S^jB'(i):1\le i\le I', 0\le j<h(i)\}$ and $\mathcal{Q}'$ is finer than $\mathscr{P}$. In addition, we can subdivide $\mathcal{P}$, into $\mathcal{P}'=\{A'(i,j)\}$ is a way such that:
	\begin{enumerate}
		\item For each $(i,j)\, \varphi[\mathbbm{1}_{S^jB'(i)}]=[\mathbbm{1}_{A'(i,j)}]$
		\item For each $(i,j)$ there exists $(\alpha,\beta)$ such that $A'(i,j)\subseteq A(\alpha,\beta)$.
	\end{enumerate}
	Additionally, Lemma~\ref{clopenpartition2} ensures that $x_0\in A'(1,h(1)-1)$ and $T(x_0)\in A'(2,0)$. Finally, a similar construction as done in the ``moreover'' part of Lemma~\ref{clopenpartition} ensures that $y_0\in S^{h(1)-1}B'(1)$ and $S(y_0)\in B'_2$.
\end{proof}

\begin{proof}[{\bf Proof of (2) $\Rightarrow$ (1)}]
Assume there is an exhaustive epimorphism $\varphi: \mathcal{G}(Y,S) \twoheadrightarrow \mathcal{G}(X,T)$ and fix points $x_0 \in X$ and $y_0$ in $Y$. Let $\{\varepsilon_k\}_{k\geq 0}$ be a sequence of positive numbers which decrease to zero such that $\varepsilon_0$ and $\varepsilon_1$ are larger than the diameters of both $X$ and $Y$.

We will recursively define a sequence of KR-partitions $\{\mathcal{Q}^n\}$ of $Y$ and KR-like partitions 
$\{\mathcal{P}^n\}$ of $X$. By requiring that every element of $\mathcal{Q}^n$ and $\mathcal{P}^n$ have diameter less than $\varepsilon_{n}$, we will insure that these partitions generate the topologies of $X$ and $Y$. 

\underline{Initial Step:}
First we define KR-partitions of height 1, 
$\mathcal{P}^1$ and $\mathcal{Q}^1$. 
Find a clopen set $B_1(1) \subset Y$ which contains $y_0$ but not $S(y_0)$. Set $B_1(2) = B_1(1)^c$ and
$\mathcal{Q}^1 = \{B_1(1),B_1(2)\}$.

Apply Lemma \ref{clopenpartition2} to create a clopen partition
$\{A_1(1,0),A_1(2,0)\}$ of $X$ such that 
$x_0 \in A_1(1,0)$, $T(x_0) \in A_1(2,0)$, and 
$\varphi[\mchi_{B_1(i,0)}] = [1_{A_1(i,0)}]$ for $i = 1,2$. 

\noindent 
\underline{Recursive Hypothesis:}
Now suppose $n \geq 1$ and we have a refining sequence of clopen KR-partitions 
$\mathcal{Q}^1 \prec \mathcal{Q}^2 \prec \cdots \prec \mathcal{Q}^n$ of $Y$ where
$$\mathcal{Q}^k = 
\{ S^j B_k(i) : 1 \leq i \leq I_k, 0 \leq j < h_k(i) \}$$
and a refining sequence of clopen KR-like partitions 
$\mathcal{P}^1 \prec \mathcal{P}^2 \prec \cdots \prec \mathcal{P}^n$ of $X$
where 
$$\mathcal{P}^k = 
\{A_k(i,j): 1 \leq i \leq I_k, 0 \leq j < h_k(i) \}$$ 
with the following properties for all $k = 1, \ldots ,n $ and 
for all $1 \leq i \leq I_k, 0 \leq j < h_k(i)$:
\begin{enumerate}
	\item $ \varphi[\mchi_{S^jB_k(i)}]=[\mchi_{A_k(i,j)}] $
	\item $\diam(A_k(i,j))<\varepsilon_k$,
	\item $\diam(S^j B_k(i))<\varepsilon_k$,
	\item $\diam \left( \cup_i A_k(i,0)\right)<\varepsilon_{k-1}$ and $\diam \left( \cup_i A_k(i,h(i)-1)\right)<\varepsilon_{k-1}$,
	\item $\diam \left( \cup_i B_k(i)\right)<\varepsilon_{k-1}$, and $\diam \left( \cup_i S^{h(i)-1}B_k(i)\right)<\varepsilon_{k-1}$,
	\item $x_0 \in A_k(1,h_k(1)-1)$ and $T(x_0) \in A_k(2,0)$,
	\item $y_0 \in S^{h(1)-1}B_k(1)$ and $S(y_0) \in B_k(2)$.
\end{enumerate}
Further assume that for $1 \leq k \leq n$, 
we have defined continuous functions $p_k: X \setminus \Top(\mathcal{P}^k) \to \mathbb{N}$ such that 
\begin{enumerate}
	\item $p_k = p_{k-1}$ on $X \setminus \Top(\mathcal{P}^{k-1})$, 
	\item the map $T_k'$ defined as $T_k'(x) = T^{p(x)}(x)$ gives a homeomorphism from 
	$A_k(i,j)$ to $A_k(i,j+1)$ for all $i$ and $0\leq j < h_k(i)-1$.
\end{enumerate}

\noindent 
\underline{Recursive Step:}	The recursive step is broken down into four parts.

\textbf{Step 1}:
Use Lemma \ref{refinetower} to create initial KR-partition for $(Y,S)$
$$\mathcal{Q} = \{S^j B(i) : 1 \leq i \leq I, 0 \leq j < h(i)\}$$
of $Y$ refining $\mathcal{Q}^n$ with the following properties for all $i,j$: 
\begin{itemize}
	\item $\diam(S^j B(i))<\varepsilon_{n+1}$,
	\item $ \cup_i B(i) \subset B_n(2)$, and 
	$\cup_i S^{h(i)-1}B(i) \subset S^{h_n(1)-1}B_n(1)$,
	\item $y_0 \in S^{h(1)-1}B(1)$ and $S(y_0) \in B(2)$.
\end{itemize}
Note that by the recursive hypothesis, the second item above implies
$\diam \left( \cup_i B(i)\right)<\varepsilon_n$, and $\diam \left( \cup_i S^{h(i)-1}B(i)\right)<\varepsilon_n$.

\textbf{Step 2}: Fix $i$ and consider the partition of $B_n(i)$ into elements 
$S^{j_1} B(i_1), S^{j_2} B(i_2), \ldots , S^{j_m} B(i_m)$ of $\mathcal{Q}$. 
By the recursive hypothesis, we know $\varphi[\mchi_{B_n(i)}]=[\mchi_{A_n(i,0)}]$.
Use Lemma \ref{clopenpartition2} to partition $A_n(i,0)$ into clopen sets 
$A(i_1,j_1), A(i_2,j_2), \ldots , A(i_m, j_m)$ such that 
$\varphi[\mchi_{S^{j_k}B(i_k)}]=[\mchi_{A(i_k,j_k)}]$ for all $k$.
In the particular case of $i=2$, use the ``Moreover...'' part of Lemma \ref{clopenpartition2} to insure $T(x_0) \in A(2,0)$. 

Set $z_0 = (T_n')^{-h_n(1)+1}(x_0) \in A_n(1,0)$. We know from Step 1 that $y_0 \in S^{h(1)-1} B(1)$. Therefore, one of the sets in the partition of $A(1,0)$ is $A(1,h(1)-h_n(1))$. Use the ``Moreover...'' part of Lemma \ref{clopenpartition2} to insure $z_0 \in A(1,h(1)-h_n(1))$. 

Consider any set $A(i,j)$ where $A(i,j) \subset A_n(k,0)$. 
For $1 \leq l < h_n(i)$, define $A(i,j+l)=(T_n')^lA(i,j)$. 
If $j+l < h(i)$, we have the following 
$$[\mchi_{A(i,j+l)}] = \varphi[\mchi_{S^{j+l}B(i)}] = \varphi[\mchi_{S^{j+l+1}B(i)}] = [\mchi_{A(i,j+l+1)}].$$
Therefore, the union of the partitions of the individual sets in $\mathcal{P}^n$ results in a KR-like partition 
$$\mathcal{P} = \{A(i,j) : 1 \leq i \leq I, 0 \leq j < h(i)\}$$
satisfying 
\begin{itemize}
	\item $\varphi[\mchi_{S^jB(i)}]=[\mchi_{A(i,j)}] $,
	\item $ \cup_i A(i,0) \subset A_n(2,0)$, and 
	$\cup_i A(i,h(i)-1) \subset A_n(1,h_n(i)-1)$,
	\item $x_0 \in A(1,h(1)-1)$ and $T(x_0) \in A(2,0)$.
\end{itemize}
Again by the recursive hypothesis, the second item above implies \\
$\diam \left( \cup_i A(i,0)\right)<\varepsilon_{n}$ and $\diam \left( \cup_i A(i,h(i)-1)\right)<\varepsilon_n$.
Let $A(i,j) \in \mathcal{P}$ be a level with $A(i,j) \subset \Top (\mathcal{P}^n)$ and $0 \leq j < h(i)$.
Because $\mathcal{P}$ is a KR-like partition, we may apply Lemma \ref{speedup} to define a continuous function $p_{n+1}$ on $A(i,j)$ so that $x \mapsto T^{p_{n+1}}(x)$ gives a homeomorphism from $A(i,j)$ to $A(i,j+1)$. For any  $A(i,j) \in \mathcal{P}$ with $A(i,j) \subset X \setminus \Top (\mathcal{P}^n)$, set $p_{n+1} = p_n$ on $A(i,j)$. Define $T'(x) = T^{p_{n+1}(x)}(x)$ on $X \setminus \Top (\mathcal{P})$. 

Thus the partitions $\mathcal{Q}$ and $\mathcal{P}$ satisfy all of the conditions (1) and (3)-(7). The next two steps are required to guarantee (2). 

\textbf{Step 3}: Use Lemma \ref{refinemirror} to refine $\mathcal{P}$ to create a KR-like partition
$\mathcal{P}^{n+1} = \{A_{n+1}(i,j) : 1 \leq i \leq I_{n+1}, 0 \leq j < h_{n+1}(i)\}$ 
where the diameter of every set $A_{n+1}(k,j)$ is less than $\varepsilon_{n+1}$, and 
$A_{n+1}(k,j) \subset A(i,j)$. If necessary, renumber the first index 
so that  $x_0 \in A_{n+1}(1,h_{n+1}(1)-1)$ and $T(x_0) \in A_{n+1}(2,0)$.

\textbf{Step 4}: In a manner similar to Step 2, 
we will use Lemma \ref{clopenpartition} to refine $\mathcal{Q}^{n}$ to 
create $\mathcal{Q}^{n+1}$ to mirror $\mathcal{P}^{n+1}$. Note that this is the one step of the proof that requires that the epimorphism be exhaustive. 

Fix $i,j$ and consider the partition of $A(i,j)$ into elements $\{A_{n+1}(k,j)\}$ of $\mathcal{P}^{n+1}$. Use Lemma \ref{clopenpartition} to partition $S^j B(i)$ into clopen sets\linebreak $\{S^j B_{n+1}(k)\}$ such that 
$\varphi[\mchi_{S^jB_{n+1}(k)}]=[\mchi_{A_{n+1}(k,j)}]$ for all $k,j$.
In the particular cases of $S^{h(1)-1} B(1)$ and $B(2)$, use the ``Moreover...'' part of Lemma \ref{clopenpartition} to insure that $y_0 \in S^{h(1)-1}B_{n+1}(1)$ and $S(y_0) \in B_{n+1}(2)$.

\noindent 
\underline{Endgame:}
With this recursive construction, we establish the existence of an infinite sequence of 
refining sequence of clopen KR-partitions $\{\mathcal{Q}^k\}_{k\geq 1}$ and clopen KR-like partitions $\{\mathcal{P}^k\}_{k\geq 1}$ satisfying properties (1)-(7) for all $k \geq 1$. For any $x \in \setminus \cap \Top(\mathcal{P}^k) = \{x_0\}$, we have defined $p(x) = \lim_{n \to \infty}p_n(x)$ and $T'(x) = T^{p(x)}(x)$. We complete the definition of $p$ by noting that property (4) implies $\cap \Top(\mathcal{P}^k) = \{x_0\}$ and 
defining $p(x_0)=1$. Because $p_k$ is continuous on $X \setminus \Top(\mathcal{P}^k)$, the point $x_0$ is the only possible point of discontinuity for $p$. 
Setting $T'(x)=T^{p(x)}(x)$ completes the definition of a strong speedup of $T$.

We now define a set map 
$h : \cup \mathcal{P}^k \to \cup \mathcal{Q}^k$ so that 
$h(A_k(i,j)) = S^jB_k(i)$. This set map respects subsets by construction of the sets
$A_k(i,j)$ and $S^jB_k(i)$. By properties (2) and (3) the sequences of partitions $\{\mathcal{Q}^k\}$ and $\{\mathcal{P}^k\}$ separate points in the spaces $X$ and $Y$ respectively; therefore, $h$ gives a homeomorphism from $X$ to $Y$ which we will also denote by $h:X\to Y$. 

Note that property (6) and (7) implies $h(x_0) = y_0$ and $h(T(x_0)) = S(y_0)$ since the partition elements that contain those points are always paired by the map $h$. 

Since $T'A_n(i,j)=A_n(i,j+1)$ for any $A_n(i,j) \in \mathcal{P}^n$ with $0\leq j < h_n(i)-1$, 
we have $$h(T'A_n(i,j)) = h(A_n(i,j+1)) = S^{j+1}B_n(i) = Sh(A_n(i,j)).$$
The point $x_0$ is the only point in $X$ which is in $A_n(i, h_n(i)-1)$ for all $n\geq 1$ and
$T(x_0)$ is the only point in $X$ which is in $A_n(i, 0)$ for all $n\geq 1$, and we have
$$h(T'(x_0)) = h(T(x_0)) = S(y_0) = Sh(y_0).$$
It follows that $T'$ and $S$ are conjugate, and that $S$ is a strong speedup of $T$, completing the proof. 

\end{proof}
\section{Topological Speedups}
In this section we work with general speedups and prove Main Theorem \ref{main}. We will follow the same overall strategy as in the previous section and prove several analogous results. Note that here, the notation $[f]$ denotes an equivalence class in $G/Inf$ instead of $G$. 

\subsection{Proof of (1) $\Rightarrow$ (3)}
\begin{theorem}
Let $(X,T)$ and $(Y,S)$ be minimal Cantor systems. If $S$ is a speedup of $T$ then there exists homeomorphism $h:X\rightarrow Y$, such that $h_{*}:M(X,T)\hookrightarrow M(Y,S)$ is an injection.
\end{theorem}

\begin{proof}
Let $p:X\rightarrow\mathbb{Z}^{+}$ be such that $T'(x)=T^{p(x)}(x)$ is a minimal homeomorophism of $X$ and let $\mu\in M(X,T)$. 
Because $T$ is invertible, $M(X,T)=M(X,T^{-1})$ thus it suffices to 
show that for any Borel set $A\subset X$ and any $\mu \in M(X,T)$ we have $\mu(T'A) = \mu (A)$. We obtain this by partitioning $A$ into sets according to the value of $p$. Since $p$ is lower semi-continuous the result follows. More specifically we have the following where $\sqcup$ denotes disjoint union: 
	\begin{align*}
	\mu(T'(A))&=\mu\left(T'\left(\displaystyle\bigsqcup_{n\in\mathbb{Z}^{+}}A\intersect p^{-1}(\{n\})\right)\right) \\
	&=\mu\left(\displaystyle\bigsqcup_{n\in\mathbb{Z}^{+}}T'(A\intersect p^{-1}(\{n\}))\right) \\
	&=\mu\left(\disUnion_{n\in\Z^{+}}T^{n}(A\intersect p^{-1}(\{n\})) \right)\\
	&=\displaystyle\sum_{n\in\mathbb{Z}^{+}}\mu(T^{n}(A\intersect p^{-1}(\{n\}))) \\
	&=\displaystyle\sum_{n\in\mathbb{Z}^{+}}\mu(A\intersect p^{-1}(\{n\}))\text{ as $\mu\in M(X,T)$.} \\
	&=\mu(A).
	\end{align*}

Now suppose $(X,T')$ and $(Y,S)$ are conjugate through $F:X \to Y$. 
Then by the above argument, the map $F_*: \mu \to \mu \circ F^{-1}$ provides an injection $F_*:M(X,T) \to M(Y,S)$. 
\end{proof}

Note that if $F_*$ above provides a bijection, then by Theorem \ref{GPS Theorem}, $T$ and $S$ would be orbit equivalent. 

\subsection{Proof of (3) $\Rightarrow$ (2)}
The following is an analog of Lemma \ref{towercob}.

\begin{lemma} \label{inflemma}
Suppose $(X,T)$ is a minimal Cantor system and $f \in C(X,\Z)$. Then $f \in \Inf(X,T)$ if and only if $\frac{1}{N} \sum_{j=0}^{N-1} fT^{j}$ converges uniformly to $0$. 
\end{lemma}

\begin{proof}
Suppose $\frac{1}{N}\sum_{j=0}^{N-1}fT^{j}$ does not converge uniformly to $0$; then there exists $\varepsilon>0$, increasing sequence $\{N_{k}\}$, and corresponding sequence $\{x_k\}$ for which $$\left|\frac{1}{N_k}\sum_{j=0}^{N_{k}-1}fT^j(x_k)\right|>\varepsilon.$$ 

Define $\mu_{k}=\sum_{j=0}^{N_{k}-1}\delta_{T^{j}x_k}$, where $\delta_z$ represents the Dirac delta measure at $z$. Because the set of Borel probability measures on $X$ form a compact set in the weak$^*$-topology, the sequence $\{\mu_{k}\}$ has a convergent subsequence, which will rename $\{\mu_{k}\}$, whose limit $\mu$ is $T$-invariant. Hence,
\begin{align*}
\left|\int f\,d\mu\right|&=\lim_{k\rightarrow\infty}\left|\int f\, d\mu_{k}\right|\\
&=\lim_{k\rightarrow\infty}\frac{1}{N_{k}}\left|\sum_{j=0}^{N_{k}-1}fT^{j}(x_k)\right|\\
&\ge\varepsilon.
\end{align*}
Therefore, $f\notin \Inf(X,T)$.

Conversely, suppose $\frac{1}{N} \sum_{j=0}^{N-1} fT^{j}$ converges uniformly to $0$ and let $\mu \in M(X,T)$ where $\mu$ is ergodic. Then by the Ergodic Theorem there exists $x \in X$ such that  
$$\int f d \mu =\lim_{N \to \infty} \frac{1}{N} \sum_{j=0}^{N-1} fT^{j}(x)$$
which implies $\int f d \mu=0$. By the Ergodic Decomposition Theorem, $f \in \Inf(X,T)$.
\end{proof}

The following is an analog of Lemma \ref{towerpos}.

\begin{lemma}\label{towerpos2}
Let $(X,T)$ be a minimal Cantor system and suppose $[f] \in G(X,T)/\Inf G(X,T)$. Then $[f] \in (G(X,T)/\Inf G(X,T))_+ \setminus \{0\}$
if and only if there is an $\varepsilon>0$ such that for any clopen set $A$, if $\diam(A)< \varepsilon$ then for all $x \in A$, 
$$\sum_{j=0}^{r_A(x)-1} fT^j(x) > 0.$$ 
\end{lemma}
\begin{proof}
Suppose $[f] \in (G(X,T)/\Inf G(X,T))_+\setminus \{0\}$, then $f$ is equal to a nonzero, nonnegative function $g$ plus an element $[k]$ of $\Inf G(X,T)$. There is a nonempty clopen set $C$ such that for $x \in C$, $g(x)>0$. Since $T$ is minimal, there is an $r>1$ such that for all $x \in X$, $\{x,T(x), \ldots T^{r-1}(x) \} \cap C \neq \emptyset$. Therefore, if $n > 2r$, we have 
$$\sum_{j=0}^{n-1} gT^j(x) \geq \frac{1}{r+1} n.$$ 
Also, by Lemma \ref{inflemma}, there is an $N$ such that if $x \in X$ and $n > N$, 
$$\left| \sum_{j=0}^{n-1} kT^j(x) \right| < \frac{1}{r+1} n.$$ 
Finally, by Lemma \ref{returntime}, there is an $\varepsilon>0$ such that if $A$ is a clopen set with $\diam (A) < \varepsilon$ then for all $x \in A$, $r_A(x)>\max \{N, 2r \}$, which implies
$$\sum_{j=0}^{r_A(x)-1} fT^j(x) > 0.$$ 

The converse follows from Lemma~\ref{towerpos}.
\end{proof}

\begin{lemma}
	\label{posbound}
	Let $(X,T)$ be a minimal Cantor system and suppose $[f] \in G(X,T)/\Inf G(X,T)$. If $[f] \in (G(X,T)/\Inf G(X,T))_+ \setminus \{0\}$
	then there is a $c>0$ such that  
	$$\int f d \mu \geq c$$ 
	for all $\mu \in M(X,T)$. 
\end{lemma}
\begin{proof}
	Suppose $[f] \in (G(X,T)/\Inf G(X,T))_+\setminus \{0\}$, then since\linebreak $G(X,T)/\Inf G(X,T)$ is simple, there is an $n\geq 1$ such that 
	$n [f] > [1]$. Therefore 
	$$n \int f d \mu \geq 1$$
	for all $\mu \in M(X,T)$. 
\end{proof}

The following is an analog of Lemma \ref{functionineq}.

\begin{lemma}\label{functionineq2}
	Suppose $[f],[g],\in G(X,T)/\Inf(G(X,T))$ with $0<[f]<[g]$. Then there exists $f'\in C(X,\Z)$ such that $0\le f'(x)\le g(x)$ for every $x\in X$ and $[f']=[f]$.
\end{lemma}
\begin{proof}
Both $[f]$ and $[g-f]$ are in $(G(X,T)/\Inf G(X,T))_+\setminus \{0\}$. Take $\varepsilon$ to be smaller than the minimum of the two numbers that exist for $f$ and $g-f$ from Lemma \ref{towerpos2}. Use Lemma \ref{refinetower} to create a KR-partition 
$\mathcal{P} = \{T^jA(i) : 1 \leq i \leq I, 0 \leq j < h(i)\}$ over a set $A$ with diameter less than $\varepsilon$ such that $g$ and $f$ are constant on every level of $\mathcal{P}$. 
Then for any $x \in A(i)$, we have 
$$\sum_{j=0}^{h(i)-1} gT^j(x) \geq \sum_{j=0}^{h(i)-1} fT^j(x) \geq 0.$$
Therefore, it is possible to select values for a new function $f'$ on levels of the $i$th column of $\mathcal{P}$ so that $g(x) \geq f'(x) \geq 0$ for all $x$ and so that $$\sum_{j=0}^{h(i)-1} f'T^j(x) = \sum_{j=0}^{h(i)-1} fT^j(x).$$
By Lemma \ref{towercob}, $[f']=[f]$. 
\end{proof}

The following is an analog of Corollary \ref{clopensubset}.

\begin{coro}[\cite{Glasner-Weiss}]\label{Glasner-Weiss}
	Let $(X,T)$ be a minimal Cantor system. Let $f\in C(X,\Z)$ satisfy $0<\int fd\mu <1$ for every $\mu\in M(X,T)$. Then there exists a clopen subset $A$ in $X$ such that $\int fd\mu=\mu(A)$ for every $\mu\in M(X,T)$. That is $[\mathbbm{1}_A]=[f]$ in $G(X,T)/\Inf(G(X,T))$. 
\end{coro}

\begin{theorem}
	Let $(X,T)$ and $(Y,S)$ be minimal Cantor systems. Suppose there exists homeomorphism $F:X\rightarrow Y$, such that $F_{*}:M(X,T)\hookrightarrow M(Y,S)$ is an injection. 
	Then there exists an exhaustive epimorphism $\varphi:\mathcal{G}(Y,S)/\Inf(G(Y,S))\twoheadrightarrow\mathcal{G}(X,T)/\Inf(G(X,T))$ which preserves the distinguished order units.
\end{theorem}
\begin{proof}
As in the proof for $(3) \implies (2)$ in Main Theorem \ref{mainstrong}, we define 
$\varphi:\mathcal{G}(Y,S)/\Inf(G(Y,S)) \to \mathcal{G}(X,T)/\Inf(G(X,T))$ by $\varphi[f] = [fF]$. Having all the analogous lemmas as in the strong speedup case, the proof is the same.
\end{proof}

\subsection{Proof of $(2)\Rightarrow (1)$}
Our strategy for this section is the same as the previous two. That is, prove the equivalent of all lemmas in the strong speedups section and use the same proof technique. Throughout, the notation $[f]$ here denotes an equivalence class in \linebreak $G(X,T)/\Inf (G(X,T))$ or $G(Y,S)/\Inf (G(Y,S))$.

Below are analogues of Lemma \ref{clopenpartition} and \ref{clopenpartition2}. The proofs are the same except we use Corollary \ref{Glasner-Weiss} in place of Corollary \ref{clopensubset}.

\begin{lemma}\label{clopenpartition-2}
	Let $(X,T)$ and $(Y,S)$ be minimal Cantor systems. Assume there is an exhaustive epimorphism $\varphi :\mathcal{G}(Y,S) \twoheadrightarrow \mathcal{G}(X,T)$. Let $A \subset X$ and $B \subset Y$ be clopen sets such that $\varphi[\mathbbm{1}_B]=[\mathbbm{1}_A]$. Then for any clopen partition $\{A_1,A_2,\ldots,A_n\}$ of $A$ there is a clopen partition $\{B_1,B_2,\ldots,B_n\}$ of $B$ such that $\varphi[\mathbbm{1}_{B_i}]=[\mathbbm{1}_{A_i}]$ for all $1 \leq i \leq n$.
	
	Moreover, given any $n$ distinct 
	points $y_1,y_2,\ldots ,y_n \in B$ 
	we may select $B_i$ to contain $y_i$ for $i=1,2,\ldots,n$.  
\end{lemma}

\begin{lemma} \label{clopenpartition2-2}
	Let $(X,T)$ and $(Y,S)$ be minimal Cantor systems. Assume there is an epimorphism $\varphi :\mathcal{G}(Y,S) \twoheadrightarrow \mathcal{G}(X,T)$. Let $A \subset X$ and $B \subset Y$ be clopen sets such that $\varphi[\mathbbm{1}_B]=[\mathbbm{1}_A]$. Then for any clopen partition $\{B_1,B_2,\ldots,B_n\}$ of $B$ there is a clopen partition $\{A_1,A_2,\ldots,A_n\}$ of $A$ such that $\varphi[\mathbbm{1}_{B_i}]=[\mathbbm{1}_{A_i}]$ for all $1 \leq i \leq n$.
	
	Moreover, given any $n$ distinct points $x_1,x_2,\ldots ,x_n \in A$ we may select $A_i$ to contain $x_i$ for $i=1,2,\ldots,n$.  
	
\end{lemma}

Unlike the previous facts, the Speedup Lemma, which would be an analogue to the Strong Speedup Lemma \ref{speedup} will take a bit more work to establish.  The following lemma serves as a precursor to the key lemma and is instrumental for its proof. 
\begin{lemma}[Partial Speedup Lemma]\label{pkey lemma}
	Let $(X,T)$ be a minimal Cantor system, and let $A,B$ be nonempty, disjoint, clopen subsets of $X$. If for all $\mu\in M(X,T),\, \mu(A)<\mu(B),$ then there exists $p:A\rightarrow\Z^{+}$ such that $T':A\rightarrow B$ defined as $S(x)=T^{p(x)}(x)$ is a homeomorphism onto its image.
\end{lemma}
\begin{proof}
	Define $f=\mathbbm{1}_{B}-\mathbbm{1}_{A}$. With the given hypothesis, \linebreak
	$[f] \in (G(X,T)/\Inf G(X,T))_+\setminus \{0\}$. Using Lemma \ref{towerpos2}, we may find a clopen set $C$ such that 
	$$\sum_{j=0}^{r_C(x)-1} fT^j(x) > 0.$$ 
	Now use Lemma \ref{refineKR} to create a KR-partition $\mathcal{P} = \{ T^j C(i) : 0 \leq i \leq I, 0 \leq j < h(i) \}$ with $\Base (\mathcal{P})  = C$ and with $\mathbbm{1}_{B}$ and $\mathbbm{1}_{A}$ constant on all levels of $\mathcal{P}$. 
	
	Now fix $i \in \{1,\ldots, I\}$. Since the number of levels in the $i$th column which are subsets of $B$ is greater than the number of levels that are subsets of $A$ there is an injection 
	$$\Gamma: \{j : 0 \leq j < h(i), T^jC(i) \subset A\}
	\to \{j : 0 \leq j < h(i), T^jC(i) \subset B\}.$$
	For $x \in T^jC_i \subset B$, set 
	$$p(x) = \min\{n > 0 : T^n(x) \in T^{\Gamma(j)}C(i) \}.$$
	As in the case of Lemma \ref{speedup}, $p$ is well-defined and continuous and $x \mapsto T^{p(x)}(x)$ is  a homeomorphism from $T^jC(i)$ to $T^{\Gamma(j)}C(i)$. By working similarly in each column, the map $x \mapsto T^{p(x)}(x)$ yields a homeomorphism from $A$ to a clopen subset of $B$.
\end{proof}
We will use our partial speedup lemma to inductively prove a full speedup lemma. 
\begin{lemma}\label{clopen1}
	Let $(X,T)$ be a minimal Cantor system. Then for every $\varepsilon>0$ there exists $\delta>0$ such that for every clopen set $A \subset X$ with $\emph{\diam}(A)<\delta$ and every $\mu\in M(X,T),$ we have $\mu(A)<\varepsilon$.
\end{lemma}
\begin{proof}
	Fix any $\mu \in M(X,T)$. 
	Choose $N \geq 1$ such that $1/N < \varepsilon$. Use Lemma \ref{returntime} to find $\delta>0$ such that if $\diam(A) < \delta$ then for all $x \in A$, $r_A(x) >N$. Since $r_A(x)>N$ for all $x \in A$, the sets 
	$A, TA, \ldots T^{N-1}A$ are all disjoint and of equal $\mu$-measure. Therefore $\mu(A)< 1/N < \varepsilon$. 
\end{proof}
We will use induction on our previous lemma to prove our key Lemma.
\begin{lemma}[Speedup Lemma - Infinitesimal Case]\label{Key Lemma}
	Let $(X,T)$ be a minimal Cantor system and let $A,B$ be nonempty disjoint, clopen subsets of $X$. If for all $\mu\in M(X,T),\, \mu(A)=\mu(B)$, then there exists $p:A\rightarrow\Z^{+}$ such that $T':A\rightarrow B$, defined as $T'(x)=T^{p(x)}(x)$, is a homeomorphism onto $B$.
\end{lemma}
\begin{proof}
	Let $A$ and $B$ be nonempty, disjoint, clopen subsets of $X$ and $x_0\in A$. Since $T$ is minimal there exists $k\in\Z^{+}$ such that $T^{k}(x)\in B$, let $y_0=T^{k}(x_0)$. 
	We will use induction to find a decreasing sequences of sets $\{A_{n}\}_{n\ge 0}$ and $\{B_{n}\}_{n\ge 0}$ such that
	$$
	\Intersect_{n\ge 0}A_{n}=\{x_0\} \text{ and }\Intersect_{n\ge 0}B_{n}=\{y_0\}
	$$
	all while defining $T'$ on larger and larger parts of $A$. 
		
	First find a clopen set $C_1\subsetneq A$ containing $x_0$ such that $\diam (C_1) < \varepsilon_1$. Set $A_1 = A \setminus C_1$. 
	Since every $\mu \in M(X,T)$ is full, $\mu(C_1)>0$ for all $\mu \in M(X,T)$ and for some $\eta_1>0$, $\mu(A_1)<\mu(B)-\eta_1$ for all $\mu \in M(X,T)$. Use Lemma \ref{clopen1} to find $D_1 \subset B$ containing $y_0$ with $\mu(D_1)< \eta_1$ for all $\mu \in M(X,T)$. 
	Therefore, we have $\mu(A_1)<\mu(B\setminus D_1)$ for all $\mu \in M(X,T)$. 
	Then by Lemma \ref{pkey lemma}, we can define a continuous $p: A_1 \to \Z^+$ so that $T'(x) = T^{p(x)}(x)$ is a homeomorphism from $A_1$ into $B \setminus D_1$. Set $B_1 = T'(A_1)$ and continue. Note that $x_0 \in A \setminus A_1$ and $y_0 \in B \setminus B_1$. We can find a $D_1 \subset B$ containing $y_0$ with $\mu(D_1)< \eta_1$ for all $\mu \in M(X,T)$. 
	
	For the next step, we start in the set $B$. We find a clopen set $D_2 \subsetneq B \setminus B_1$ of $y_0$ such that $\diam(D_2) < \varepsilon_2$. Set $B_2 = B \setminus (B_1 \cup D_2)$ Then using an argument similar to the above, we can apply Lemma \ref{clopen1} to find $C_2 \subset A\setminus A_1$ containing $x_0$ such that $\mu(A \setminus C_2)>\mu(B_2)$ for all $\mu$. Apply Lemma \ref{pkey lemma} to $T^{-1}$, $B_2$ and $A \setminus C_2$ we can define a continuous 
	$q: A_1 \to \Z^+$ so that 
	$y \mapsto T^{-q(y)}(y)$ is a homeomorphism from $B_2$ into $A \setminus C_2$. 
	Set $A_2$ equal to the image of this homeomorphism. If $x\in A_2$ then set $p(x) = q(y)$ where $x$ is the image of $y$. Then $T'(x) = T^{p(x)}(x)$ is a homeomorphism from $A_2$ to $B_2$. 
	
	We continue recursively to define sets $A \supset A_1 \supset A_2 \supset  \cdots$ and $B \supset B_1 \supset B_2 \supset \cdots $ such that 
	\begin{itemize}
		\item $\diam(A_n) < \varepsilon_n$ for $n$ odd, 
		\item $\diam(B_n) < \varepsilon_n$ for $n$ even, 
		\item $x_0 \in A_n$ for all $n$, 
		\item $y_0 \in B_n$ for all $n$,
		\item $\mu(A_n) = \mu(B_n)$ for all $n$. 		
	\end{itemize}
	Moreover, as we are defining $A_n$ and $B_n$ we are defining $p:A_n \setminus A_{n-1} \to \N$ such that $T': (A \setminus A_n) \to (B \setminus B_n)$ defined by $T'(x) = T^{p(x)}(x)$ is a homeomorphism. We complete the definition of $T'$ on $A$ by setting $p(x_0) = k$ so that $T'(x_0) = T^k(x_0)=y_0$. 
	
	Note that $T'$ is continuous at $x_0$ by the following argument. Let $\varepsilon>0$ be given. Find $n$ such that $\varepsilon_n < \varepsilon$. 
	Choose $\delta>0$ such that $B(x_0, \delta) \subset A_{n}$.  Then we have 
	$$T'(B(x_0,\delta)) \subset T'(A_{n}) = B_n \subset B(y_0,\epsilon_n)
	\subset B(y_0,\epsilon).$$ This finishes the proof. 
\end{proof}

With the key lemma now extended to speedups in general, we may exactly mimic the proof of $(2)\Rightarrow (1)$ from the strong speedups sections. 

\begin{theorem}
	Let $(X,T)$ and $(Y,S)$ be minimal Cantor systems. If there exists an exhaustive epimorphism $\varphi:\mathcal{G}(Y,S)/\Inf(G(Y,S))\twoheadrightarrow\mathcal{G}(X,T)/\Inf(G(X,T))$ which preserves the distinguished order units then $(Y,S)$ is a speedup of $(X,T)$.
\end{theorem}

\begin{proof}
	The proof follows exactly the same argument as for (2) $\implies$ (1) in Section 4. The only differences are the following: 
	\begin{itemize}
		\item Lemma~\ref{clopenpartition-2} is used in place of Lemma~\ref{clopenpartition} 
		\item Lemma~\ref{clopenpartition2-2} is used in place of Lemma~\ref{clopenpartition2} 
		\item the Speedup Lemma (Lemma~\ref{Key Lemma}) is used in place of the Strong Speedup Lemma (Lemma~\ref{speedup})
		\item Lemma~\ref{refinemirror} is true using Lemma~\ref{clopenpartition2-2} instead of Lemma~\ref{clopenpartition2}. 
	\end{itemize}
	Note that the argument from (2) $\implies$ (1) in Section 4 which discusses the continuity properties of $p$ does not follow here because Lemma~\ref{Key Lemma} does not imply that $p$ is continuous on elements of the clopen KR-partitions. However, for the conclusion of this theorem, we do not need to control for the continuity of the function $p$. 
\end{proof}

\section{Speedup Equivalence}
A natural question arises in light the results of speedups and orbit equivalence. If two minimal Cantor systems are speedups of one another, are they orbit equivalent? To discuss this further, define minimal Cantor systems $(X_{1},T_{1})$ and $(X_{2},T_{2})$ to be \emph{{speedup equivalent}} if $T_{1}$ is a speedup of $T_{2}$ and $T_{2}$ is a speedup of $T_{1}$. Combining \cite[Thm 2.2]{GPS} with our Main Theorem 2 we obtain the following corollary.

\begin{coro}\label{Cor.1}
	Let $(X_{1},T_{1})$ and $(X_{2},T_{2})$ be minimal Cantor systems. If $(X_{1},T_{1})$ and $(X_{2},T_{2})$ are orbit equivalent, then $(X_{1},T_{1})$ and $(X_{2},T_{2})$ are speedup equivalent.
\end{coro}

Rephrasing Corollary~\ref{Cor.1} above, as equivalence relations orbit equivalence is contained in speedup equivalence. This leads us to a fundamental question: are orbit equivalence and speedup equivalence the same equivalence relation? In this section, we prove the orbit equivalence and speedup equivalence are the same when each minimal Cantor system has finitely many ergodic measures. In general, orbit equivalence and speedup equivalence are not the same as shown by Julien Melleray's forthcoming paper \emph{Dynamical Simplicies and Fra\"iss\'e Theory} \cite{Melleray}. Before we proceed, we will need the following proposition. 
\begin{prop}\label{MutuallySingular}
	Let $(X_{i},T_{i})$ be minimal Cantor systems and $\varphi:X_{1}\rightarrow X_{2}$ be a homeomorphism. If $\varphi_{*}:M(X_{1},T_{1})\hookrightarrow M(X_{2},T_{2})$ is an injection, then $\varphi_{*}$ preserves pairs of mutually singular measures.
\end{prop}	
We now use Proposition~\ref{MutuallySingular} to prove the following theorem.
\begin{theorem}\label{speedupequivalence}
	Let $(X_{i},T_{i}),\, i=1,2,$ be minimal Cantor systems each with finitely many ergodic measures. If $(X_{1},T_{1})$ and $(X_{2},T_{2})$ are speedup equivalent, then $(X_{1},T_{1})$ and $(X_{2},T_{2})$ are orbit equivalent.
\end{theorem}
\begin{proof}
	Since $T_{1}$ and $T_{2}$ are speedup equivalent, combining part $(3)$ of Main Theorem 1 and Proposition~\ref{MutuallySingular} it follows immediately that 
	$$
	|\partial_{e}(M(X_{1},T_{1}))|=|\partial_{e}(M(X_{2},T_{2}))|
	$$ 
	and without loss of generality we may assume $|\partial_{e}(M(X_{1},T_{1}))|=n$ for some $n\in\Z^{+}$. Every measure $\mu$ in $M(X_{2},T_{2})$ is a convex combination of ergodic measures in a unique way. With this in mind for $\mu\in M(X_{2},T_{2})$, let $E(\mu)$ denote the collection of all ergodic measures of $M(X_{2},T_{2})$ which have a positive coefficient in the unique ergodic decomposition of $\mu$. Observe if $\mu_{1},\mu_{2}\in M(X_{2},T_{2})$ with $\mu_{1}\neq\mu_{2}$ and $\mu_{1}\perp\mu_{2}$, then
	$$
	E(\mu_{1})\intersect E(\mu_{2})=\emptyset.
	$$
	Now as $T_{2}$ is a speedup of $T_{1}$,  
	there exists $\varphi:X_{1}\rightarrow X_{2}$, a homeomorphism, such that 
	$$
	\varphi_{*}:M(X_{1},T_{1})\hookrightarrow M(X_{2},T_{2})
	$$
	is an injection. Since $|\partial_{e}(M(X_{1},T_{1}))|=|\partial_{e}(M(X_{2},T_{2}))|=n$ and $\varphi_{*}$ is injective, we see that $\{E(\varphi_{*}(\nu_{i}))\}_{i=1}^{n}$, where $\{\nu_{i}\}_{i=1}^{n}=\partial_{e}(M(X_{1},T_{1}))$, is a collection of $n$ pairwise disjoint sets, as distinct ergodic measures are mutually singular. It follows that for each $i=1,2,\dots,n$, $E(\nu_{i})$ is a distinct singleton, and therefore $\varphi_{*}(\partial_{e}(M(X_{1},T_{1})))=\partial_{e}(M(X_{2},T_{2}))$. Coupling the facts that $\varphi_{*}$ is an affine map and a bijection on extreme points, we may conclude that $\varphi_{*}$ is a bijection, and hence is an affine homeomorphism between $M(X_{1},T_{1})$ and $M(X_{2},T_{2})$ arising from a space homeomorphism. Therefore, by \cite[Thm. $2.2$]{GPS} $(X_{1},T_{1})$ and $(X_{2},T_{2})$ are orbit equivalent.
\end{proof}                                                            

\begin{note}
	In \cite{AAO} the authors give an example of a unique ergodic systems whose square is minimal and has more than one invariant measure. Hence, speedups can leave not only the conjugacy class but orbit equivalence class as well. 
\end{note}

\end{document}